\documentclass[reqno]{amsart}

\usepackage[utf8]{inputenc}
\usepackage[T1]{fontenc}
\usepackage{graphicx}
\usepackage{longtable}
\usepackage{wrapfig}
\usepackage{rotating}
\usepackage[normalem]{ulem}
\usepackage{amsmath}
\usepackage{amssymb}
\usepackage{capt-of}
\usepackage{hyperref}
\author[A.~D.~Chopra]{Alakh Dhruv Chopra}
\author[F.~Pakhomov]{Fedor Pakhomov}
\usepackage{amsmath}
\usepackage{mathtools}
\usepackage{microtype}
\usepackage{cleveref}
\usepackage{xparse}
\usepackage{caption}
\usepackage{subcaption}
\date{}
\title{Well-quasi-orders on finite trees and transfinite sequences}
\hypersetup{
 pdfauthor={},
 pdftitle={},
 pdfkeywords={},
 pdfsubject={},
 pdflang={English}}
\begin{document}

\newtheorem{theorem}{Theorem}
\numberwithin{theorem}{section}
\newtheorem{corollary}[theorem]{Corollary}
\newtheorem{lemma}[theorem]{Lemma}
\newtheorem{proposition}[theorem]{Proposition}

\theoremstyle{definition}
\newtheorem{definition}[theorem]{Definition}
\newtheorem{examples}[theorem]{Example} 
\newtheorem{question}[theorem]{Question}
\newtheorem{conjecture}[theorem]{Conjecture}
\newtheorem{assumption}[theorem]{Standing Assumption}
\newtheorem*{claim}{Claim}

\theoremstyle{remark}
\newtheorem*{note}{Note}
\newtheorem{remark}[theorem]{Remark}
\newtheorem{notation}[theorem]{Notation}

\newenvironment{subproof}{\begin{proof}[Subproof]}{\end{proof}}

\providecommand\given{}
\newcommand\SetSymbol[1][]{%
\nonscript\:#1\vert
\allowbreak
\nonscript\:
\mathopen{}}
\DeclarePairedDelimiterX\Set[1]\{\}{%
\renewcommand\given{\SetSymbol[\delimsize]}
#1
}

\newcommand{\Pf}{\mathsf{P_f}}
\newcommand{\Tf}{\mathsf{T_f}}
\newcommand{\Tfs}[1]{\mathsf{T_{f,#1}}}
\newcommand{\wdesc}{\omega^2_{\text{desc}}}
\newcommand{\quasiembedding}{\hookrightarrow}
\newcommand{\bbvarphi}{\bar{\bar{\varphi}}}
\newcommand{\leaf}[1]{\cdot{#1}}
\newcommand{\unlabelledvertex}[1]{\cdot(#1)}
\newcommand{\labelledvertex}[2]{\dot{#1}(#2)}
\NewDocumentCommand{\vertex}{ m g }{\IfNoValueTF{#2}{\unlabelledvertex{#1}}{\labelledvertex{#1}{#2}}}
\newcommand{\factors}{\text{factors}}
\newcommand{\incomp}{\mathrel{\bot}}
\begin{abstract}
We study the well-quasi-order (wqo) consisting of the set of finite trees with
leaf labels coming from an arbitrary wqo \(Q\), ordered by tree homomorphisms
which respect the order on the labels. This is a variant of the usual Kruskal
tree ordering without infima preservation. We calculate the precise maximal
order types of this class of wqos as a function of the maximal order type of the
labels \(Q\). In the process, we sharpen some recent results of Friedman and
Weiermann \cite{Friedman2023}.
Furthermore, we show a correspondence with indecomposable transfinite sequences
with finite range, over elements of the wqo \(Q\), of length less than
\(\omega^\omega\). Nash-Williams proved that \emph{arbitrary} transfinite sequences with
finite range are also well-quasi-ordered \cite{NashWilliams1965}, but there are
no known methods to extract bounds on the maximal order type from the proof.
More concrete proofs for sequences of length less than \(\alpha\) for some \(\alpha
< \omega^\omega\) were given by Erdős and Rado \cite{Erdos1959}. Using the
correspondence, we obtain precise bounds for the entire collection of
transfinite sequences with finite range of length less than \(\omega^\omega\).
\end{abstract}

\maketitle

\begingroup
\thispagestyle{empty}

\renewcommand{\footnoterule}{}
\renewcommand{\thefootnote}{}
\setlength{\skip\footins}{0pt}
\setlength{\footnotesep}{0pt}
\footnotetext{Published in the Bulletin of Symbolic Logic with DOI: \href{https://doi.org/10.1017/bsl.2026.10172}{10.1017/bsl.2026.10172}.}
\endgroup
\section{Introduction}
\label{sec:introduction}
The theory of well-quasi-orders (wqos) is well-studied -- and often
rediscovered, according to Kruskal \cite{Kruskal1972} -- in many fields of
mathematics, computer science, and logic. A good reference that tries to capture
some of this scope is \cite{WQO2020}. Amongst multiple equivalent
characterizations of wqos is the following: \((Q,\le_{Q})\) is a wqo if and only
if every linearization (i.e., a superset of \(\le_Q\) such that every pair of
elements is comparable) of \(Q\) is a well-order. Thus, we can look at the
supremum of all such ordinals which are order types of linearization of a given
quasi-order (qo) \(Q\). Let us notate it as \(o(Q,\le_Q)\), or simply \(o(Q)\). In
their seminal paper back in 1977 \cite{deJongh1977}, de Jongh and Parikh showed
that this supremum is in fact a \emph{maximum}: that there is some linearization of
\((Q,\le_Q)\) such that its order type is equal to \(o(Q)\). The ordinal \(o(Q)\) can
thus be viewed as a measure of the strength of the wqo \(Q\), and is called its
\emph{maximal order type}, or simply its \emph{type}. The canonical references for the
calculation of maximal order types of various wqos are de Jongh and Parikh's
founding paper \cite{deJongh1977} as well as Schmidt's thesis
\cite{DianaSchmidt2020}. In \cref{sec:well-quasi-orders}, we review the theory of
wqos and their maximal order types in more detail.

A classical construction for a qo \(Q\) is the set of its finite sequences \(Q^*\)
ordered by sequence embeddability, and is also known as the \emph{Higman order}
\cite{Higman1952} over \(Q\). It is known to be a wqo when \(Q\) is a wqo, and its
maximal order type \(o(Q^*)\) has been calculated as a function of \(o(Q)\) in
\cite[Theorem 2.9]{DianaSchmidt2020}. The natural generalization is to consider
sequences of \emph{transfinite} length. However, this was shown to fail immediately
by Rado \cite{Rado1954}: if \(Q\) is a wqo, the qo \(Q^\omega\) of sequences of
length \(\omega\) is not necessarily a wqo. In order to preserve the property of
being a wqo, we need to consider transfinite sequences with \emph{finite range}
(i.e., that use only a finite number of elements).

For a qo \(Q\), let \(s^F_\alpha(Q)\) be the qo of transfinite sequences with finite
range of length less than \(\alpha\), ordered by sequence embeddability.
Nash-Williams showed that \(s^F_\alpha(Q)\) is a wqo for any \(\alpha\) if \(Q\) is a
wqo \cite{NashWilliams1965}, but his proof is non-constructive and offers no
immediate information on its maximal order type. Earlier, Erdös and Rado
\cite{Erdos1959} showed the same for any specific \(\alpha < \omega^\omega\) with
a much more constructive proof. We extend their results to
\(s^F_{\omega^\omega}(Q)\). The main goal of this paper is to compute the maximal
order type \(o(s^F_{\omega^\omega}(Q))\) as a function of \(o(Q)\).

In order to do so, we first consider a certain qo on finite leaf-labelled trees.
More specifically, finite trees with labels only on the leaf vertices (i.e.,
vertices with no children), with the labels coming from another specified qo.
Two such trees, say \(\tau_0\) and \(\tau_1\), are ordered by the existence of a map
from the vertices of \(\tau_0\) to the vertices of \(\tau_1\) that preserves the
internal tree structure (i.e., the ancestor-descendant relation) and respects
the underlying ordering on the leaf labels. Such maps are called \emph{tree
homomorphisms}. This qo, which we denote by \(\Tf(Q)\), is a wqo whenever \(Q\) is a
wqo. It is a weakening of the classical \emph{Kruskal's tree ordering}, where the
maps are \emph{tree homeomorphisms} that additionally preserve infima (w.r.t. the
internal tree structure). To our knowledge, it was first studied by Montalbán
\cite{Montalbn2006} in the context of so-called \emph{signed trees}. A similar qo
with an even weaker ordering also appears in works by Selivanov
\cite{Selivanov2004}. More recently, Friedman and Weiermann
\cite{Friedman2023} considered the maximal order type of this qo (in a more
general setting), giving non-tight upper and lower bounds in terms of
\(o(Q)\).\footnote{The authors learnt about their work while working on
\cref{sec:leaf-labelled-trees}.} In \cref{sec:leaf-labelled-trees}, we discuss this qo in more detail
and calculate its precise order type (stated in \cref{thm:Tf-order-type}).

Now we consider a particularly uniform suborder of \(s^F_\alpha(Q)\): the qo of
\emph{indecomposable sequences} which we denote by \(i^F_\alpha(Q)\). In
\cref{sec:transfinite-sequences}, we define this qo, and show that \(\Tf(Q)\) and
\(i^F_{\omega^\omega}(Q)\) are \emph{essentially} the same qo. This correspondence is
stated more precisely in \cref{subsec:tree-sequence-correspondence} using some
back-and-forth maps, which ultimately culminates in
\cref{thm:tree-sequence-correspondence}. This lets us reuse the results for
\(\Tf(Q)\) to state the maximal order type for \(i^F_{\omega^\omega}(Q)\) (stated in
\cref{thm:iww-order-type}), and by using the relation between indecomposable
sequences and arbitrary sequences, also the maximal order type for
\(s^F_{\omega^\omega}(Q)\) (stated in \cref{thm:sww-order-type}). This is the main
contribution of the paper. Some remarks on similar lines of research in recent
years and further work can be found in \cref{sec:further-questions}.
\section{Well-quasi-orders}
\label{sec:well-quasi-orders}
In this section, we briefly review the theory of wqos and their associated
ordinal invariants, in particular the maximal order type. Schmidt's thesis
\cite{DianaSchmidt2020} and de Jongh and Parikh's founding paper
\cite{deJongh1977} are the canonical references, while the recent volume on
wqos \cite{WQO2020} contains a modern treatment of the subject (albeit with
sometimes differing terminology).
\subsection{Orders and maps}
\label{subsec:orders-and-maps}
We start by looking at qos in general and various kinds of maps between them.

\begin{definition}
A \textbf{quasi-order} (qo) (or a \textbf{pre-order}) is a pair \((Q,\le)\) consisting of the
underlying set \(Q\) and a binary relation \(\mathord{\le} \subseteq Q \times Q\)
such that \(\le\) is:
\begin{enumerate}
\item reflexive: \(\forall x. \, \, x \le x\), and
\item transitive: \(\forall x. \, \forall y. \, \forall z. \, \, (x \le y) \land (y
   \le z) \implies x \le z\).
\end{enumerate}
For \(x,y \in Q\), if \(x \le y\) and \(y \not\le x\), we write \(x < y\). If both \(x
\not\le y\) and \(y \not\le x\), we say that \(x\) and \(y\) are \emph{incomparable} and
denote it as \(x \incomp y\). If there are no incomparable elements, we call it a
\textbf{total} (or \textbf{linear}) order.
\end{definition}

We usually identify both the qo and its underlying set by the same symbol, for
example \(Q\) in the above definition, with the meaning clear from context.

Note that, unlike a \textbf{partial-order} (po) the relation \(\le\) may not be
anti-symmetric: \(x \le y \land y \le x \land x \neq y\) is possible. The
distinction is not very important, since taking the quotient of a qo under the
equivalence relation \(x \equiv y\) iff \((x \le y) \land (y \le x)\) transforms it
into a po.

Sometimes it is useful to look at a subset \(P \subseteq Q\) of a qo \((Q,\le)\), in
which case the qo \((P, \le \upharpoonright_{P \times P})\) is called an \textbf{induced
suborder} of \(Q\).

We often compare qos by defining maps between their underlying sets such that
the orderings of the domain and the codomain are related in some sense. In the
following, let us fix two arbitrary qos \((P,\le_P)\) and \((Q,\le_Q)\).

\begin{definition}
Let \(h \colon P \to Q\) be a map. If \(x \le_P y\) implies \(h(x) \le_Q h(y)\) for
all \(x,y \in P\), then \(h\) is called an \textbf{order-preserving map}. If the converse
holds -- that is, if \(h(x) \le_Q h(y)\) implies \(x \le_P y\) for all \(x,y \in P\)
-- then \(h\) is called an \textbf{order-reflecting map} (or a \textbf{quasi-embedding}) and we
denote it by \(P \quasiembedding Q\).

The map \(h\) is an \textbf{order-embedding} if it is both an order-preserving and
order-reflecting map.
\end{definition}

When a map is an order-embedding, the domain and range of the map are
\emph{essentially} the same qos, modulo renaming of elements and equivalences. For
the following definition, assume \(P\) and \(Q\) to be pos, that is, satisfying
anti-symmetry.

\begin{definition}
A \emph{surjective} order-embedding \(h \colon P \to Q\) is called an
\textbf{order-isomorphism}: that is, \(h\) is a bijection, and \(x \le_P y\) iff \(h(x)
\le_Q h(y)\) for all \(x,y \in P\). If such a map exists, we say that \(P\) and \(Q\)
are \textbf{order-isomorphic} and denote it by \(P \cong Q\).
\label{def:po-equivalent}
\end{definition}

The condition of surjectivity has to be modified slightly to make it work for
qos. Recall that, for \(x,y \in Q\), \(x \equiv_Q y\) iff \((x \le_Q y) \land (y
\le_Q x)\). (Here we use subscripts to differentiate \(\equiv_Q\) and \(\equiv_P\).)

\begin{definition}
Qos \(P\) and \(Q\) are \textbf{order-isomorphic} (or \textbf{equivalent}) if there are
order-embeddings \(f \colon P \to Q\) and \(g \colon Q \to P\) such that:
\begin{enumerate}
\item for any \(p \in P\), \(g(f(p)) \equiv_P p\);
\item for any \(q \in Q\), \(f(g(q)) \equiv_Q q\).
\end{enumerate}
We denote this by \(P \cong Q\) as well.
\label{def:qo-equivalent}
\end{definition}

The two definitions are almost the same: if \(f\) and \(g\) are modified to be maps
between \(P/{\equiv_P}\) and \(Q/{\equiv_Q}\) (i.e., their respective quotients
under the equivalence relations \(\equiv_P\) and \(\equiv_Q\)), then they are both
order-isomorphisms in the sense of \cref{def:po-equivalent}, and functional
inverses of each other.

\begin{remark}
We can weaken the definition to have both \(f,g\) be either order-preserving or
order-reflecting maps, and for either \(g(f(p)) \equiv_P p\) for all \(p \in P\) or
\(f(g(q)) \equiv_Q q\) for all \(q \in Q\) to hold. It is not too difficult to prove
that it remains the same as the definition above.

For example, let \(f,g\) be order-embeddings, and \(g(f(p)) \equiv_P p\) for all \(p
\in P\). Pick \(q \in Q\). Then \(g(q) \in P\) and \(g(f(g(q))) \equiv_P g(q)\). Since
\(g\) is an order-embedding, it follows that \(q \equiv_Q f(g(q))\).
\label{remark:qo-equivalent}
\end{remark}

The notion of being order-isomorphic is transitive via function composition of
the associated maps.
\subsection{Well-quasi-orders}
\label{subsec:well-quasi-orders}
The central notion of our study is the concept of a wqo.

\begin{definition}
A qo \((Q,\le)\) is a well-quasi-order (wqo) if every infinite sequence of
elements \((x_i)_{i \in \omega}\) contains a weakly-increasing pair \(x_i \le x_j\)
for some \(i < j < \omega\).

Any sequence of elements \((x_i)_{i < k}\) which does \emph{not} contain such a
weakly-increasing pair is called a \emph{bad sequence}, and for a wqo must
necessarily be finite.
\end{definition}

\begin{remark}
A qo \((Q,\le)\) is also often defined to be a wqo by the existence of the
following two properties:
\begin{enumerate}
\item \emph{Well-foundedness}: There is no infinite descending sequence \(x_0 > x_1 > x_2
   > \dots\) of elements of \(Q\).
\item \emph{Finite antichain property}: There is no infinite set of elements \((x_i)_{i <
   \omega}\) such that they are pairwise incomparable (i.e., \(x_i \incomp x_j\)
for all \(i < j < \omega\)).
\end{enumerate}
However, in the base arithmetical theory of \(\mathsf{RCA}_0\), the existence of
these two properties does not imply the definition above \cite[Theorem 2.9]{Marcone2020}.
\end{remark}

The notion of a \textbf{well-partial-order} (wpo) is analogous. If the po in question
is a total order, we call it a \textbf{well-order}. The canonical examples of
well-orders are, of course, the ordinals.

Wqos have the so-called \emph{finite basis property}; that is, any non-empty subset
\(S\) of a wqo \(Q\) has a finite non-zero number of minimal elements. This easily
follows from the definition above, and is in fact equivalent to it.

The property of being a wqo is quite well-behaved when it comes to
order-reflecting maps. Fix arbitrary qos \((P, \le_P)\) and \((Q, \le_Q)\).

\begin{proposition}
If \(P \quasiembedding Q\) then \(Q\) is a wqo implies that \(P\) is also a wqo.
\end{proposition}

\begin{proof}
Let \(h \colon P \to Q\) be an order-reflecting map. Consider an infinite sequence
\((x_i)_{i < \omega}\) of elements of \(P\). Since \(Q\) is a wqo, its image
\((h(x_i))_{i < \omega}\) must have a weakly-increasing pair, say \(h(x_j) \le_Q
h(x_l)\) for some \(j < l < \omega\). Then, by the definition of order-reflecting
maps, we have a weakly-increasing pair \(x_j \le_P x_l\) in the original sequence.
\end{proof}

Thus, order-reflecting maps are very useful tools in the study of wqos, in
particular, by reducing the question of one qo being a wqo to known, often
simpler, examples.
\subsection{Ordinal invariants}
\label{subsec:ordinal-invariants}
There is a third characterization of wqos in terms of their \emph{linearizations},
which we shall use heavily in this paper.

\begin{definition}
A \textbf{linearization} (or linear extension) of a qo \((Q,\le)\) is a linear (or total)
qo \((Q, \le^+)\) such that \(\mathord{\le^+} \supseteq \mathord{\le}\), that is,
\(\forall x. \, \forall y. \, \, x \le y \implies x \le^+ y\).
\end{definition}

Note that \((Q, \le^+) \quasiembedding (Q, \le)\) with the identity map, thus if
\(Q\) is a wqo then so are all of its linearizations.

\begin{proposition}
If a qo \((Q,\le)\) is a wqo, then all its possible linearizations are well-orders
(modulo equivalence of elements).
\end{proposition}

Since linearizations of qos may not satisfy anti-symmetry, we implicitly take
the quotient under equivalence when describing them as well-orders.

\begin{remark}
Technically speaking, the converse direction -- a qo is a wqo if all possible
linearizations are well-orders -- is another characterization of wqos that does
not imply our stated definition in the base system \(\mathsf{RCA}_0\)
\cite[Theorem 2.9]{Marcone2020}.
\end{remark}

All well-orders have an associated \emph{order type}, defined to be the unique
ordinal that they are order-isomorphic to. A natural object to study in the
context of wqos, thus, is the following.

\begin{definition}
The \textbf{maximal order type} of a wqo \((Q,\le)\) (or simply its \textbf{order type}) is the
supremum of the order types of all its possible linearizations. It is denoted by
\(o(Q,\le)\), or simply \(o(Q)\).
\end{definition}

The study of maximal order types of wqos was initiated by de Jongh and Parikh
\cite{deJongh1977}, where they also show that the maximal order type is in fact
a \emph{maximum} (i.e., it is attained as the order type of some linearization). A
comprehensive survey is contained in the habilitation thesis of Schmidt
\cite{DianaSchmidt2020}; it also contains the erroneous calculation of
\(o(s^F_\alpha(Q))\), which this paper tries to partially resolve.

Another equivalent characterization is the following.

\begin{proposition}
For a wqo \(Q\), \(o(Q) = \sup \Set{\alpha \in \mathsf{Ord} \given \alpha
\quasiembedding Q}\), where \(\mathsf{Ord}\) is the class of all ordinals.
\label{prop:order-type-sup-quasiembedding}
\end{proposition}

This helps us state one of the most useful tools in the calculation of maximal
order types.

\begin{proposition}
For wqos \(P\) and \(Q\), if \(P \quasiembedding Q\) then \(o(P) \leq o(Q)\).
\end{proposition}

\begin{proof}
For some ordinal \(\alpha\), assume \(\alpha \quasiembedding P\). Since \(P
\quasiembedding Q\), we get \(\alpha \quasiembedding Q\) by function composition.
The statement follows from \cref{prop:order-type-sup-quasiembedding}.
\end{proof}

\begin{corollary}
If \(P \cong Q\) then \(o(P) = o(Q)\).
\end{corollary}

As before, this lets us give upper bounds for the maximal order type of a wqo
under consideration, with the help of a potentially larger but better understood
wqo.

The maximal order type of a wqo can also be studied inductively by certain
downwards-closed sets generated by its elements.

\begin{definition}
Let \((Q,\le)\) be a qo. For any element \(x \in Q\), we define the set \(L_Q(x) =
\Set{ y \in Q \given x \not\le y }\). If \(Q\) is a wqo, the induced suborder
\((L_Q(x), \le\upharpoonright_{L_Q(x) \times L_Q(x)})\) is also a wqo, with the
maximal order type \(l_Q(x) = o(L_Q(x))\).

As always, we denote it as \(L(x)\) and \(l(x)\) when the qo \(Q\) is clear from
context.
\label{def:lower-sets}
\end{definition}

The following theorem is one of the main results in \cite{deJongh1977}.

\begin{theorem}
For a wqo \((Q,\le)\), the maximal order type \(o(Q)\) is equal to \(\sup_{x \in Q}
(l(x) + 1)\) (or simply \(\sup_{x \in Q} l(x)\) if \(o(Q)\) is a limit ordinal).
\label{thm:lower-sets-sup}
\end{theorem}

There are two other ordinal measures of a given wqo that are also studied: the
\emph{height} and the \emph{width}. Together with the maximal order type, these three are
referred to as the \emph{ordinal invariants} of the wqo. Their definitions are also
given inductively by narrowing the definition of the downward-closed sets.

\begin{definition}
Let \((Q,\le)\) be a qo. For any element \(x \in Q\), we define the sets \(L_{<,Q}(x)
= \Set{ y \in Q \given y < x }\) and \(L_{\incomp,Q}(x) = \Set{ y \in Q \given x
\incomp y }\) (or just \(L_<(x)\) and \(L_{\incomp}(x)\)).
\end{definition}

\begin{definition}
For a wqo \((Q,\le)\), the \textbf{height} \(h(Q)\) is defined to be \(\sup_{x \in Q}
(h(L_<(x)) + 1)\), and the \textbf{width} \(w(Q)\) is defined to be \(\sup_{x \in Q}
(w(L_{\incomp}(x)) + 1)\).
\end{definition}

Clearly, for a wqo \(Q\) and some \(x \in Q\), the sets \(L_<(x)\) and
\(L_{\incomp}(x)\) are subsets of \(L(x)\), and thus the height and width are at
most the maximal order type. The maximal order type, though, is bounded above by
their natural product, giving the following elegant set of inequalities.

\begin{proposition}
For a wqo \(Q\), we have \(\sup\{h(Q),w(Q)\} \leq o(Q) \leq h(Q) \otimes w(Q)\).
\end{proposition}

This was originally shown in \cite[Theorem 4.13]{Kriz1990}; a self-contained
proof can also be found in \cite[Theorem 3.8]{Dzamonja2020}.

It is not difficult to see that, if two wqos \(P\) and \(Q\) are equivalent, then
all three of their ordinal invariants are equal.

As a rule of thumb, for most classes of wqos that are studied, the maximal order
type ends up being a multiplicatively indecomposable ordinal with the height
being relatively small, and thus -- by the inequalities above -- the width equal
to the maximal order type.
\subsection{Examples}
\label{subsec:examples}
The simplest examples of wqos are finite qos, since any infinite sequence of
elements will necessarily contain repetitions. Let us fix notations for finite
chains and antichains.

\begin{definition}
Finite numbers \(1,2,3,\dots\) stand for finite chains (or total orders), while
barred finite numbers \(\bar{1}, \bar{2}, \bar{3}, \dots\) stand for finite
antichains with the appropriate number of incomparable elements.
\end{definition}

Ordinals, and well-orders in general, are trivially wqos.

In order to define more complicated qos, we can construct them from multiple
simpler qos. There are many constructions that are commonly used, but the
\emph{disjoint union} and \emph{Cartesian product} are the most relevant.

Let us fix arbitrary qos \((P,\le_P)\) and \((Q,\le_Q)\).

\begin{definition}
The \emph{disjoint union} of \(P\) and \(Q\) is denoted by \(P \sqcup Q\), and is the pair
\((P \cup Q, \le_{\sqcup})\) such that for \(x,y \in P \cup Q\), \(x \le_{\sqcup} y
\iff (x,y \in P \land x \le_P y) \lor (x,y \in Q \land x \le_Q y)\).

We implicitly assume that \(P\) and \(Q\) are disjoint, but this can also be
simulated set-theoretically by taking the domain to be \((P \times \Set{0}) \cup
(Q \times \Set{1})\).
\end{definition}

\begin{definition}
The \emph{Cartesian product} of \(P\) and \(Q\) is denoted by \(P \times Q\), and is the
pair \((P \times Q, \le_{\times})\) such that for \((x_0,y_0),(x_1,y_1) \in P
\times Q\), \((x_0,y_0) \le_{\times} (x_1,y_1) \iff (x_0 \le_P x_1) \land (y_0
\le_Q y_1)\).
\end{definition}

Both constructions preserve the property of being a wqo: if \(P\) and \(Q\) are wqos
then so are \(P \sqcup Q\) and \(P \times Q\). Their maximal order types can be
described using the maximal order types of \(P\) and \(Q\) \cite{deJongh1977}.

\begin{proposition}
If \(P\) and \(Q\) are wqos, then \(o(P \sqcup Q) = o(P) \oplus o(Q)\) and \(o(P \times
Q) = o(P) \otimes o(Q)\), where \(\oplus\) and \(\otimes\) stand for the \emph{natural
sum} and \emph{natural product} respectively.
\label{prop:union-product-order-type}
\end{proposition}

Recall that the \emph{natural sum} and \emph{natural product} of ordinals (also referred
to as the \emph{Hessenberg sum} and \emph{Hessenberg product}) are both commutative,
unlike regular ordinal sum and product. A precise definition can be found in
\cite{deJongh1977}.
\subsubsection{Finite powerset}
\label{subsubsec:finite-powerset}
An important construction in the context of this paper is that of the \emph{finite
powerset} of a qo.

\begin{definition}
The finite powerset of a qo \((Q,\le)\) is the set of all its finite subsets
\(\Set{ s \subseteq Q \given \vert s \vert < \omega }\) ordered by a so-called
\emph{domination ordering} \(\le_P\), such that for finite subsets \(s\) and \(t\), \(s
\le_P t \iff \forall x \in s. \, \exists y \in t. \,\, x \le y\). This qo is
denoted by \(\Pf(Q)\).
\label{def:finite-powerset}
\end{definition}

\emph{Remark.} \(\Pf(Q)\) technically contains the empty set \(\emptyset\), but it will
be convenient to occasionally ignore its existence in the later sections. Since
it only adds a bottom element, this will be quite harmless in our context.

Another way to describe \(\Pf(Q)\) is as \emph{finitely generated downwards-closed
sets} of \(Q\): the finite subset \(s \in \Pf(Q)\) stands for the subset \(S = \Set{
y \in Q \given \exists x \in s. \,\, y \le x }\) of \(Q\). The domination ordering
\(\le_P\) can now be seen as the usual subset ordering between these
downwards-closed sets.

If \(Q\) is a wqo then so is \(\Pf(Q)\). The maximal order type, however, is no
longer a simple function of the maximal order type of \(Q\). For example, for some
well-order \(\alpha\), any element of \(\Pf(\alpha)\) is just equal to its largest
member, so we get \(o(\Pf(\alpha)) = 1 + \alpha\) after including the empty set.
However, for some finite antichain \(\bar{n}\), we get \(o(\Pf(\bar{n})) = 2^n\).
These are, in fact, the proper bounds for its maximal order type in general.

\begin{lemma}
For any wqo \((Q,\le)\), we have \(1 + o(Q) \leq o(\Pf(Q)) \leq 2^{o(Q)}\).
\label{lemma:finite-powerset-type}
\end{lemma}

The upper bound was formerly a folklore result, but it was recently formalized
by Abriola et al. \cite[Theorem 3.3]{Abriola23}.

The value of \(o(\Pf(Q))\) greatly depends on the antichain structure inside \(Q\).
In particular, if the width of \(Q\) is infinite, we can ensure that the width of
\(\Pf(Q)\) also grows exponentially \cite[Theorem 3.11]{Abriola23}.

\begin{lemma}
Given a wqo \((Q,\le)\) with \(w(Q) \geq \omega\), we have \(w(\Pf(Q)) \geq
2^{w(Q)}\).
\label{lemma-infinite-width}
\end{lemma}

Note that \(2^\omega\) is still \(\omega\), so we need a width of at least \(\omega +
1\) for it to actually increase.
\subsubsection{Finite multiset ordering}
\label{subsubsec:finite-multiset-ordering}
An auxiliary construction we will need is that of the set of \emph{finite multisets}
of a qo. Formally, a finite multiset of a set \(Q\) is a map \(m:Q\to\omega\) such
that the set \(\Set{ q \in Q \given m(q) > 0 }\) is finite. However, this
definition is quite cumbersome to use, so we will use the informal version: a
finite multiset is a finite subset of \(Q\) in which elements can appear multiple
times.

\begin{definition}
The qo \(M(Q)\) consists of the set of finite multisets of a qo \((Q,\le)\) ordered
by the so-called \emph{term ordering} \(\le_m\), such that for finite multisets \(s\) and
\(t\), \(s \le_m t\) iff there is an injective function \(f: s \to t\) such that \(x
\le f(x)\) for all \(x \in s\).
\label{def:finite-multisets}
\end{definition}

If \(Q\) is a wqo, then so is \(M(Q)\). The maximal order type \(o(M(Q))\) can be
stated precisely in terms of \(o(Q)\) (see \cite[definition 7
and theorem 5]{VanDerMeeren2014}, where this qo is denoted as \(M^{\diamond}(Q)\)).

\begin{lemma}
Given a non-empty wqo \((Q,\le)\), we have \(o(M(Q)) = \omega^{\widehat{o(Q)}}\)
(where the definition of \(\widehat{o(Q)}\) can be found in
\cite[definition 9 and notation 1]{VanDerMeeren2014}).
\label{lemma-multiset}
\end{lemma}

The crucial property is that \(\omega^{\widehat{\alpha}} > \alpha\) for any
ordinal \(\alpha\), so \(o(M(Q)) > o(Q)\).
\subsubsection{Kruskal's tree ordering}
\label{subsubsec:kruskals-tree-ordering}
The final example of qo we will mention is that of finite trees, typically also
labelled by elements from some qo, ordered by \emph{homeomorphic tree embeddability}.
This was primarily studied by Joseph Kruskal, and hence is also referred to as
the \emph{Kruskal tree ordering} \cite{Kruskal1960}. We will be studying a variant
of this ordering in this paper, so it will be discussed in more detail in the
next section.
\section{Leaf-Labelled Trees}
\label{sec:leaf-labelled-trees}
In this section, we introduce the qo consisting of \emph{finite leaf-labelled trees}
ordered by \emph{tree homomorphisms}. By comparison with Kruskal's tree ordering, an
immediate consequence is: whenever the labels come from a wqo, the qo of
leaf-labelled trees is a wqo as well. The main result of this section is the
precise calculation of its maximal order type and some reverse mathematical
corollaries.
\subsection{Finite leaf-labelled trees}
\label{subsec:finite-leaf-labelled-trees}
\begin{definition}
Let \((Q,\le_Q)\) be a non-empty qo. The class of \textbf{finite leaf-labelled trees}
\(\Tf(Q,\le_Q)\), or simply \(\Tf(Q)\), over \(Q\) is defined to be the \emph{smallest}
class such that:
\begin{enumerate}
\item \(\leaf{q} \in \Tf(Q)\) for all \(q \in Q\);
\item \(\vertex{\tau_0,\tau_1,\dots,\tau_{k-1}} \in \Tf(Q)\) for any finite set of
trees \(\tau_0, \tau_1, \dots, \tau_{k-1} \in \Tf(Q)\) with \(k \geq 1\).
\end{enumerate}

For a tree \(\tau \in \Tf(Q)\), the set of \textbf{vertices} \(v(\tau)\) is defined
inductively as:
\begin{enumerate}
\item for \(\tau = \leaf{q}\) for some \(q \in Q\), \(v(\tau) \coloneq \Set{\tau} =
   \Set{\leaf{q}}\);
\item for \(\tau = \vertex{\tau_i}_{i<k}\) for some \(k < \omega\), \(v(\tau) \coloneq
   \Set{\tau} \cup \bigcup_{i<k} v(\tau_i)\).
\end{enumerate}
Equivalently, \(v(\tau)\) is the transitive closure of \(\tau\) (under the usual
\emph{belongs-to} relation). Vertices of the form \(\leaf{q}\) are called \textbf{leaves}, or
\textbf{leaf vertices}, and are labelled by \(q\). The other vertices are \textbf{internal
vertices} and are unlabelled. The vertex equal to the whole tree is the \textbf{root},
or the \textbf{root vertex}.

We can define an \textbf{internal tree ordering} \(\sqsubset_\tau\) on the vertices of a
tree \(\tau\) as: for \(\tau_0,\tau_1 \in v(\tau)\), \(\tau_0 \sqsubset_\tau \tau_1
\iff \tau_0 \neq \tau_1 \land \tau_1 \in v(\tau_0)\). From a graph-theoretic
perspective, \(\tau_0\) is an ancestor of \(\tau_1\), or equivalently, \(\tau_0\) lies
on the unique path from \(\tau_1\) to the root vertex. The order
\(\sqsubset_{\tau}\) is a strict (non-reflexive) po. Leaf vertices are maximal in
the \(\sqsubset_\tau\) ordering, while the root is the minimum element.

The \textbf{height} \(ht(\tau)\) of a tree \(\tau \in \Tf(Q)\) is defined inductively as:
\begin{enumerate}
\item for \(\tau = \leaf{q}\) for \(q \in Q\), \(ht(\tau) \coloneq 0\);
\item for \(\tau = \vertex{\tau_i}_{i<k}\) for \(k<\omega\), \(ht(\tau) \coloneq 1 +
   \max_{i<k} ht(\tau_i)\).
\end{enumerate}
The height of a tree in \(\Tf(Q)\) is always finite. See
\cref{fig:leaf-labelled-trees-examples} for examples.
\end{definition}

\begin{figure}[t]
\begin{subfigure}{0.55\textwidth}
\begin{minipage}[b]{0.25\textwidth}
\centering
\includegraphics[page=1]{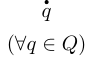}
\caption*{Leaf vertices}
\end{minipage}
\hfill
\begin{minipage}[b]{0.65\textwidth}
\centering
\includegraphics[page=2]{figures.pdf}
\caption*{Internal vertices}
\end{minipage}
\caption{The two types of vertices in $\Tf(Q)$.}
\end{subfigure}
\hfill
\begin{subfigure}{0.4\textwidth}
\centering
\includegraphics[page=3]{figures.pdf}
\caption{A tree in $\Tf(2)$ of height 4.}
\end{subfigure}
\caption{Examples of finite leaf-labelled trees.}
\label{fig:leaf-labelled-trees-examples}
\end{figure}

For the rest of the subsection, let us fix an arbitrary \emph{non-empty} qo
\((Q,\le_Q)\). Unless stated otherwise, \(Q\) is not necessarily a wqo.

As of now, in the above definition, we have not used the ordering on the set of
labels \(Q\). This will come into play in the comparison of two leaf-labelled
trees.

\begin{definition}
For \(\sigma, \tau \in \Tf(Q)\), a \textbf{tree homomorphism} from \(\sigma\) to \(\tau\)
that respects the order on leaf labels is a map \(h \colon v(\sigma) \to v(\tau)\)
such that:
\begin{enumerate}
\item if \(s, t \in v(\sigma)\) such that \(s \sqsubset_\sigma t\), then \(h(s)
   \sqsubset_\tau h(t)\);
\item if \(s \in v(\sigma)\) is a leaf vertex \(\leaf{x}\), then \(h(s)\) is a leaf
vertex \(\leaf{y}\) such that \(x \le_Q y\).
\end{enumerate}
\end{definition}

We use this notion of tree homomorphisms to define an ordering on the collection
of finite leaf-labelled tree.

\begin{definition}
For \(\sigma, \tau \in \Tf(Q)\), \(\sigma \le_T \tau\) iff there is a tree
homomorphism from \(\sigma\) to \(\tau\).
\end{definition}

It is obvious from the definition that \(\le_T\) is reflexive. If \(h \colon
v(\tau_0) \to v(\tau_1)\) and \(h' \colon v(\tau_1) \to v(\tau_2)\) are tree
homomorphisms from \(\tau_0\) to \(\tau_1\) and \(\tau_1\) to \(\tau_2\) respectively,
then it easy to check that \(h' \circ h \colon v(\tau_0) \to v(\tau_2)\) is a tree
homomorphism from \(\tau_0\) to \(\tau_2\). Thus \(\le_T\) is transitive as well, and
\((\Tf(Q), \le_T)\) is a qo. See \cref{fig:tree-homomorphisms-examples} for
examples.

\begin{definition}
For a non-empty qo \(Q\), the qo consisting of the leaf-labelled trees on \(Q\)
ordered by existence of tree homomorphisms is denoted by \((\Tf(Q), \le_T)\), or
\(\Tf(Q)\) in short.
\end{definition}

\begin{figure}[ht]
\begin{subfigure}[b]{0.45\textwidth}
\centering
\includegraphics[page=4]{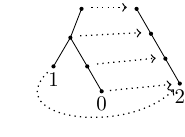}
\caption{A tree homomorphism in $\Tf(2)$.}
\end{subfigure}
\hfill
\begin{subfigure}[b]{0.45\textwidth}
\centering
\includegraphics[page=5]{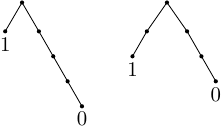}
\caption{Two incomparable trees in $\Tf(2)$.}
\end{subfigure}
\caption{Examples of tree homomorphisms.}
\label{fig:tree-homomorphisms-examples}
\end{figure}

This qo is a variant of Kruskal's tree ordering (introduced in
\cref{subsubsec:kruskals-tree-ordering}), which is defined via the notion of
\emph{homeomorphic embeddings}.

\begin{definition}
For \(\sigma, \tau \in \Tf(Q)\), an (inf-preserving) \textbf{homeomorphic tree embedding}
from \(\sigma\) to \(\tau\) is a tree homomorphism \(h \colon v(\sigma) \to v(\tau)\)
such that:
\begin{enumerate}
\item for all \(s, t \in v(\sigma)\), we have \(h(s \wedge t) = h(s) \wedge h(t)\),
where \(s \wedge t\) denotes the \emph{infimum} of \(s\) and \(t\) under the internal
tree orderings \(\sqsubset_{\sigma}\) and \(\sqsubset_{\tau}\).
\end{enumerate}
If such a map exists from \(\sigma\) to \(\tau\), we say that \(\sigma\) is
\textbf{homeomorphically embeddable} in \(\tau\).
\end{definition}

This condition, also called \emph{infimum preservation}, is the only addition in
comparison to the earlier tree homomorphisms. Note that it automatically makes
the map injective.

\begin{definition}
The Kruskal tree ordering \(\le_K\) is defined on \(\Tf(Q)\) as follows. For trees
\(\sigma, \tau \in \Tf(Q)\), \(\sigma \le_K \tau\) iff \(\sigma\) is homeomorphically
embeddable in \(\tau\). The Kruskal tree (quasi-)order is denoted by \((\Tf(Q),
\le_K)\).
\end{definition}

Technically speaking, Kruskal's tree ordering is usually defined on finite trees
with both internal and leaf labels. Adapted to the above qo, Kruskal proved the
following eponymous tree theorem \cite{Kruskal1960}.

\begin{theorem}[Kruskal's tree theorem for $\Tf(Q)$]
For any non-empty qo \(Q\), if \(Q\) is a wqo, then so is \((\Tf(Q), \le_K)\).
\end{theorem}

Bounds for the maximal order type of this ordering (for a slightly more general
collection of trees) were given by Schmidt in her survey
\cite{DianaSchmidt2020}.

It is quite easy to see that \((\Tf(Q), \le_T) \quasiembedding (\Tf(Q), \le_K)\),
which immediately gives us the following theorem.

\begin{theorem}
For any non-empty qo \(Q\), if \(Q\) is a wqo, then so is \((\Tf(Q), \le_T)\).
\end{theorem}

The notion of ordering trees via the existence of tree homomorphisms was
introduced by Montalbán \cite{Montalbn2006}, and studied further with Marcone
\cite{Marcone2009}, in the context of so-called \emph{signed trees}. Specifically,
their collection of trees were finite trees with internal vertices labelled by
\(\bar{2}\), for which they calculated the precise maximal order type (of
\(\varphi(2,0)\)). Friedman and Weiermann \cite{Friedman2023} generalized their
construction to consider finite trees with internal vertices and leaf vertices
labelled by disjoint qos; in their notation, \(\Tf(Q)\) would be denoted by
\(\mathbb{T}(1,Q)\). They calculate upper and lower bounds for the maximal order
types of these wqos, but their lower bounds are quite loose. In the next
subsection, we state the precise value of the maximal order type of \(\Tf(Q)\) as
a function of \(o(Q)\); in particular, by improving their lower bound.

An inductive definition of \(\le_T\) will be quite useful for the upcoming
sections.

\begin{definition}
For \(\sigma, \tau \in \Tf(Q)\), \(\sigma \le_T \tau\) iff:
\begin{enumerate}
\item \(\sigma = \leaf{x}\) and \(\tau = \leaf{y}\) for \(x,y \in Q\), and \(x \le_Q y\).
\item \(\sigma = \leaf{x}\) and \(\tau = \vertex{\tau_j}_{j<l}\) for \(x \in Q\), \(l <
   \omega\), and there exists \(j<l\) such that \(\leaf{x} \le_T \tau_j\).
\item \(\sigma = \vertex{\sigma_i}_{i<k}\) and \(\tau = \vertex{\tau_j}_{j<l}\) for
\(k,l < \omega\), and for all \(i<k\) there exists \(j<l\) such that \(\sigma_i
   \le_T \tau_j\).
\end{enumerate}
\end{definition}

It is an easy exercise to show that the two definitions of \(\le_T\) are
equivalent.

Compare this inductive definition of \(\le_T\) with the finite powerset operator
\(\Pf\) (see \cref{def:finite-powerset}). Informally, the set of leaf-labelled trees
\(\Tf(Q)\) can be thought of as the closure of \(Q \cup \Pf(Q) \cup \Pf(\Pf(Q))
\cup \Pf(\Pf(\Pf(Q))) \cup \cdots\) under a suitable ordering. An explicit
correspondence that we will often use is the following.

\begin{lemma}
Let \(S \subset \Tf(Q)\) be any subset of trees, and define \(\bar{S} \coloneqq
\Set{ \vertex{\tau_i}_{i<k} \given \tau_i \in S \text{ for } i < k < \omega }\).
Then, \(\Pf(S) \cong (\bar{S}, \le_T)\).

In particular, \(\Pf(\Tf(Q)) \cong \overline\Tf(Q) = (\Tf(Q) \setminus \Set{
\leaf{q} \given q \in Q })\).
\label{lemma:finite-powerset-trees}
\end{lemma}
\subsection{Maximal order types}
\label{subsec:finite-leaf-labelled-trees-type}
Since the set of leaf-labelled trees \(\Tf(Q)\) is a wqo when \(Q\) is a wqo, we
would like to calculate the maximal order type \(o(\Tf(Q))\) as some function of
\(o(Q)\).

When \(Q\) contains only a single element, this calculation is quite easy. It also
serves as a good moment to introduce notation for non-branching trees with a
single leaf vertex, whose only parameters are their height and the label on the
leaf vertex.

\begin{definition}
For \(q \in Q\), \(\leaf{q}^{k}\) for \(0 < k < \omega\) is a tree in \(\Tf(Q)\) which
contains a single leaf vertex labelled by \(q\) at height \(k-1\). Concretely:
\(\leaf{q}^1\) is \(\leaf{q}\), and \(\leaf{q}^{k+1}\) is the tree
\(\vertex{\leaf{q}^k}\).
\end{definition}

This gives us the maximal order type when the qo \(Q\) only has one equivalence
class.

\begin{lemma}
\(o(\Tf(1)) = \omega\).
\label{lemma:Tf1-bound}
\end{lemma}

\begin{proof}
It is easy to see that for any tree \(\tau \in \Tf(1)\), we have \(\tau \le_T
\leaf{0}^{ht(\tau) + 1}\) and \(\leaf{0}^{ht(\tau)+1} \le_T \tau\). This gives us
an order-embedding between \(\Tf(1)\) and \(\omega\) using the height function.
\end{proof}

The value of \(o(\Tf(2))\), however, is not so obvious. As we shall see, there is
a great jump in complexity as soon as there are just two non-equivalent
elements.

In order to calculate \(o(\Tf(Q))\) for some arbitrary wqo \(Q\), we shall first
find appropriate upper bounds for its value, then prove that they serve as a
lower bound as well.
\subsubsection{Upper bounds}
\label{subsubsec:upper-bounds}
The usual way to show upper bounds for the maximal order type of wqos is by
studying the structure of the sets \(L(x)\) defined by its elements (recall
\cref{def:lower-sets} and \cref{thm:lower-sets-sup}). Instead of calculating their
precise maximal order types, they are embedded into simpler wqos using
order-reflecting maps, whose order types are easier to describe.

The upper bounds for the value of \(o(\Tf(Q))\) will be given using the \emph{epsilon
numbers} \(\varepsilon_\alpha\), with \(\alpha\) ranging over the ordinals. The
epsilon numbers enumerate the fixed-points of \(\omega\text{-exponentiation}\):
that is, every epsilon number \(\varepsilon\) satisfies \(\omega^\varepsilon =
\varepsilon\). This enumeration is, of course, strictly increasing.

In particular, the first epsilon number \(\varepsilon_0\) is the supremum of the
ordinals \(\Set{ \omega, \omega^\omega, \omega^{\omega^\omega},
\omega^{\omega^{\omega^\omega}}, \dots }\).

The choice of epsilon numbers as the upper bound is because of their closure
properties, which we state without proof.

\begin{proposition}
Let \(\varepsilon\) be an epsilon number, and \(\delta, \gamma < \varepsilon\) be
two arbitrary ordinals smaller than \(\varepsilon\). Then:
\begin{enumerate}
\item \(\varepsilon\) is closed under \(\omega\text{-exponentiation}\): \(\omega^\delta
   < \varepsilon\).
\item \(\varepsilon\) is closed under finite sums: \(\delta + \gamma <
   \varepsilon\).
\item \(\varepsilon\) is closed under finite products: \(\delta \cdot \gamma <
   \varepsilon\).
\end{enumerate}
\label{prop:epsilon-numbers-closure-properties}
\end{proposition}

Ordinals closed under finite sums and products are called \emph{additively
indecomposable ordinals} and \emph{multiplicatively indecomposable ordinals},
respectively.

The enumeration of the epsilon numbers also has fixed points: ordinals \(\delta\)
satisfying the equality \(\varepsilon_\delta = \delta\). This causes a problem
because of the following inequality.

\begin{proposition}
Given a non-empty wqo \(Q\), \(o(\Tf(Q)) > o(Q)\).
\end{proposition}

\begin{proof}
Consider \(L(\leaf{0}^2)\), where \(0\) is any minimal element of \(Q\). Clearly it
contains all the leaf vertices, and thus \(o(Q) \leq l(\leaf{0}^2) < o(\Tf(Q))\).
\end{proof}

Thus, the upper bound cannot be given by a function with fixed points, so we
modify the enumeration of the epsilon numbers slightly.

\begin{definition}
The \emph{fixed-point-free epsilon function} \(\bar\varepsilon\) is defined as follows:
$$ \bar\varepsilon(\alpha) = \begin{cases}
\varepsilon_{\alpha+1}
& \text{if $\alpha = \gamma + n$ for some $\gamma$ and $n < \omega$ such that }
\varepsilon_{\gamma} = \gamma, \\
\varepsilon_{\alpha}
& \text{otherwise.}
\end{cases} $$
\label{def:fixed-point-free-epsilon-function}
\end{definition}

We can now define the function we will use to upper bound the value of
\(o(\Tf(Q))\).

\begin{definition}
Let \(g\) be a function from ordinals to ordinals defined as follows:
$$ g(\alpha) = \begin{cases}
0 & \alpha = 0, \\
\omega & \alpha = 1,\\
\bar\varepsilon(-2 + \alpha) & \alpha \geq 2.
\end{cases}
$$
\label{def:function-g}
\end{definition}

Another way to describe the function \(g\) is as follows. Let \(g(0) = 0\). Then,
\(g(\alpha)\) for \(\alpha > 0\) is the least ordinal strictly greater than \(\alpha\)
and \(g(\beta)\) for any \(\beta < \alpha\) such that \(g(\alpha)\) is closed under
2-exponentiation: for any \(\gamma < g(\alpha)\), \(2^{\gamma} < g(\alpha)\). Note
that \(g\) is fixed-point-free (for \(\alpha > 0\)).

The change from \(\omega\text{-exponentiation}\) to 2-exponentiation is to
incorporate \(\omega\) into the sequence. Recall that \(2^\omega = \omega\) and thus
\(\omega\) is also a fixed point of 2-exponentiation, and the rest of the fixed
points of 2-exponentation and \(\omega\text{-exponentiation}\) coincide.

The function \(g\) inherits the closure properties of the epsilon numbers, while
being fixed-point-free and strictly increasing. Since these properties are
crucial to the calculation of the upper bound, let us restate them as a
proposition.

\begin{proposition}
The function \(g\) satisfies the following properties:
\begin{enumerate}
\item Every ordinal in \(\text{range}(g)\) is closed under 2-exponentiation, finite
sums, and finite products.
\item \(g\) is fixed-point-free: for any ordinal \(\alpha > 0\), \(\alpha \neq
   g(\alpha)\).
\item \(g\) is strictly increasing: for ordinals \(\alpha < \beta\), \(g(\alpha) <
   g(\beta)\).
\end{enumerate}
\label{prop:function-g-closure-properties}
\end{proposition}

This finally lets us state and prove the upper bound for arbitrary wqos.

\begin{proposition}
Given a non-empty wqo \((Q,\le_Q)\), \(o(\Tf(Q)) \leq g(o(Q))\).
\label{prop:Tfq-upperbound}
\end{proposition}

\begin{proof}
We induct on the order type \(o(Q)\). For the rest of the proof, we denote
\(L_{\Tf(Q)}(\tau)\) by \(L(\tau)\) for \(\tau \in \Tf(Q)\), and keep the subscript
for \(L_Q(x)\) for \(x \in Q\).

The case for \(o(Q) = 1\) is immediate. Consider an arbitrary wqo \(Q\) with \(o(Q) >
1\), with the statement holding for all smaller wqos. We show by induction on the
height \(ht(\tau)\) that \(o(L(\tau)) < g(o(Q))\) for all \(\tau \in \Tf(Q)\).

Let \(\tau\) be a leaf vertex, that is, \(\tau = \leaf{x}\) for some \(x \in Q\). If
\(x\) is the minimum element of \(Q\) then \(L(\tau) = \emptyset\), so assume
otherwise. A tree \(\sigma\) belongs to \(L(\tau)\) iff there is no leaf vertex
\(\leaf{y} \in v(\sigma)\) such that \(x \le_Q y\), and hence all the leaf vertices
of \(\sigma\) are labelled by elements of \(L_Q(x) \subset Q\). Thus \(L(\tau) \cong
\Tf(L_Q(x))\), and $$ o(L(\tau)) \leq o(\Tf(L_Q(x))) \leq g(o(L_Q(x))) < g(o(Q))
$$ since \(o(L_Q(x)) < o(Q)\) and \(g\) is a strictly monotonic function.

Otherwise, let \(\tau = \vertex{\tau_i}_{i<k}\) with \(k<\omega\). Every leaf vertex
\(\leaf{x}\) for \(x \in Q\) belongs to \(L(\tau)\) by definition. On the other hand,
let \(\sigma \in L(\tau)\) be a tree \(\vertex{\sigma_j}_{j<l}\). Since \(\tau
\not\le_T \sigma\), there is a \(\tau_i\) with \(i<k\) such that for every \(\sigma_j\)
with \(j<l\) we have \(\tau_i \not\le_T \sigma_j\). In this case \(\sigma\) can be
described as a member of \(\Pf(L(\tau_i))\) (recall
\cref{lemma:finite-powerset-trees}).

Thus, $$ L(\tau) \quasiembedding \Set{\leaf{x} \given x \in Q} \sqcup
\bigsqcup_{i<k} \Pf(L(\tau_i)). $$

Then,
\begin{align*}
o(L(\tau))
& \leq o(Q) \oplus \bigoplus_{i<k} o(\Pf(L(\tau_i))) \\
& \leq o(Q) \oplus \bigoplus_{i<k} 2^{o(L(\tau_i))}
\end{align*}
where we get the bounds on the order types using
\cref{prop:union-product-order-type} and \cref{lemma:finite-powerset-type}.

We know that \(o(Q) < g(o(Q))\), by \cref{prop:function-g-closure-properties}. Since
\(ht(\tau_i) < ht(\tau)\), we can apply the induction hypothesis to get
\(o(L(\tau_i)) < g(o(Q))\) for each \(i<k\). By
\cref{prop:function-g-closure-properties}, we also know that \(g(o(Q))\) is closed
under finite sums and \(2\text{-exponentation}\). Thus, we get \(o(L(\tau)) <
g(o(Q))\).

Since \(o(\Tf(Q)) = \sup_{\tau\in\Tf(Q)}o(L(\tau))\), we get the upper bound
\(o(\Tf(Q)) \leq g(o(Q))\).
\end{proof}

\begin{remark}
The proof is, unsurprisingly, quite similar to the one presented in
\cite[Lemma 6.4]{Friedman2023}. However, when \(o(Q)\) is finite, their upper
bound is slightly looser: if \(o(Q) = k < \omega\), their upper bound would give
\(o(\Tf(Q)) \leq \varepsilon_k\) instead of the sharper \(o(\Tf(Q)) \leq
\varepsilon_{k-2}\) (or \(\omega\) when \(k=1\)) proved above.
\end{remark}
\subsubsection{Lower bounds}
\label{subsubsec:lower-bounds}
Our next goal is to show that the function \(g\) also works as a lower bound for
the maximal order type. We do so by considering the \emph{width} instead, and show a
lower bound for \(w(\Tf(Q))\). For that, we will use the connection to finite
powersets along with the exponential growth provided by
\cref{lemma-infinite-width}.

First, let us revisit the connection to finite powersets in the context of this
section.

\begin{lemma}
For any non-empty wqo \(Q\), \(o(\Tf(Q)) = o(\Pf(\Tf(Q)))\) and
\(w(\Tf(Q)) = w(\Pf(\Tf(Q)))\).
\label{prop:tf-pf-order-type}
\end{lemma}

\begin{proof}
This follows from \cref{lemma:finite-powerset-trees}, since the map \(\tau \mapsto
\vertex{\tau}\) is an order-embedding from \(\Tf(Q)\) to \((\Tf(Q) \setminus \Set{
\leaf{q} \given q \in Q }) = \overline\Tf(Q)\), and \(\overline\Tf(Q) \cong
\Pf(\Tf(Q))\).
\end{proof}

Now, using \cref{lemma-infinite-width}, we can already state something quite
strong.

\begin{lemma}
Let \(Q\) be a non-empty wqo. If \(w(\Tf(Q)) > \omega\), then
\(w(\Tf(Q))\) is an epsilon number.
\label{prop:tf-width-epsilon}
\end{lemma}

\begin{proof}
Let \(w(\Tf(Q)) \geq \alpha > \omega\). By \cref{lemma-infinite-width},
\(w(\Pf(\Tf(Q))) \geq 2^\alpha\). Then, since \(w(\Pf(\Tf(Q))) = w(\Tf(Q))\) by
\cref{prop:tf-pf-order-type}, \(w(\Tf(Q)) \geq 2^\alpha\). Thus \(w(\Tf(Q))\) is
closed under 2-exponentiation, and since the fixed points of 2-exponentation and
\(\omega\)-exponentiation greater than \(\omega\) coincide, it is also closed under
\(\omega\)-exponentation.
\end{proof}

Finally, consider an arbitrary wqo \(Q\) and its maximal order type \(o(Q)\).
Clearly, \(\Tf(o(Q)) \quasiembedding \Tf(Q)\) by lifting the map for \(o(Q)
\quasiembedding Q\). This implies that \(o(\Tf(o(Q))) \leq o(\Tf(Q))\), so any
lower bound for \(o(\Tf(o(Q)))\) is also a lower bound for \(o(\Tf(Q))\). Thus, for
this section, it will suffice to consider only well-orders as labels.

The first non-trivial calculation is for the case of two labels.

\begin{lemma}
\(w(\Tf(2)) \geq \omega + 1\).
\label{lemma:Tf2-lowerbound-width}
\end{lemma}

\begin{proof}
Let \(\wdesc\) be an induced suborder of \(\omega \times \omega\) consisting of the
set of all tuples \((i,j)\) with \(j < i < \omega\). We will show that \(\wdesc\) can
be order-embedded into \(\Tf(2)\).

Clearly \(\wdesc\) is a wqo. In fact, its maximal order type is also \(\omega^2\),
since for any fixed \(k \in \omega\) we can show that \(\omega \cdot k
\quasiembedding \wdesc\) (by mapping \(\omega \cdot j + i < \omega \cdot k\) for \(j
< k\) and \(i < \omega\) to the tuple \((i+k, j) \in \wdesc\)). The width of \(\wdesc\)
is at least \(\omega\) since we can construct arbitrarily long antichains of the
form \(\{(2k,0), (2k-1,1), (2k-2,2), \dots, (k+1,k-1)\}\) for any finite \(k > 0\).

Let \(h \colon \wdesc \to \Tf(2)\) be defined as follows (see
\cref{fig:map-h-omega2-Tf2}). Fix \((i,j) \in \wdesc\) with \(i > j\). Then, \(h((i,j))
= \vertex{\leaf{0}^{i+3},\vertex{\leaf{1},\leaf{0}^{j+2}}}\). This tree does not
collapse: \(\leaf{0}^{i+3} \not\le_T \vertex{\leaf{1},\leaf{0}^{j+2}}\) since
\(ht(\leaf{0}^{i+3}) = i + 2 > j + 2 = ht(\vertex{\leaf{1},\leaf{0}^{j+2}})\), and
\(\vertex{\leaf{1},\leaf{0}^{j+2}} \not\le_T \leaf{0}^{i+3}\) since
\(\leaf{0}^{i+3}\) has no leaf vertex labelled by \(1\).

\begin{figure}[h]
\begin{subfigure}{0.4\textwidth}
\centering
\includegraphics[page=6]{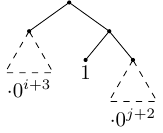}
\end{subfigure}
\hfill
\begin{subfigure}{0.4\textwidth}
\centering
\includegraphics[page=7]{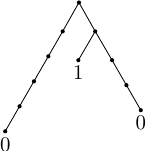}
\end{subfigure}
\caption{The map $h \colon \wdesc \to \Tf(2)$ for arbitrary $(i,j) \in
\wdesc$, and an example of the tree $h((2,1))$.}
\label{fig:map-h-omega2-Tf2}
\end{figure}

It is easy to see that \((i_0,j_0) \le (i_1,j_1)\) iff \(h((i_0,j_0)) \le_T
h((i_1,j_1))\), and thus \(h\) is an order-embedding. The tree \(\leaf{1}^4\) is
incomparable with the range of \(h\) (since the minimum height of a tree in its
range, specifically \(h((1,0))\), is at least 4), and thus \(w(\Tf(2)) \geq
\omega + 1\).
\end{proof}

\begin{corollary}
\(o(\Tf(2)) \geq w(\Tf(2)) \geq \varepsilon_0 = g(2)\).
\label{corr:Tf2-lowerbound}
\end{corollary}

\begin{proof}
By \cref{prop:tf-width-epsilon} and \cref{lemma:Tf2-lowerbound-width}, \(w(\Tf(2))\)
is an epsilon number. The smallest epsilon number is \(\varepsilon_0\), and thus
the statement follows.
\end{proof}

We now generalize the proof for ordinals \(\gamma > 2\), with the only problematic
case being the fixed points of the enumeration of the epsilon numbers (i.e., for
\(\gamma = \varepsilon_\gamma\)).

\begin{lemma}
For any ordinal \(\gamma > 2\), \(o(\Tf(\gamma)) \geq w(\Tf(\gamma)) \geq
\bar\varepsilon(-2 + \gamma) = g(\gamma)\).
\label{lemma:Tfq-lowerbound}
\end{lemma}

\begin{proof}
We prove this by induction on \(\gamma\), with \(\gamma = 2\) being the base case.
We have already shown that \(o(\Tf(2)) \geq w(\Tf(2)) \geq \varepsilon_0\).

Let \(\gamma = \alpha + 1\) be a successor ordinal, with \(o(\Tf(\alpha)) \geq
w(\Tf(\alpha)) \geq \bar\varepsilon(-2 + \alpha)\) by the induction hypothesis.
Let \(S \coloneqq \Set{ \vertex{\tau} \given \tau \in \Tf(\alpha) }\). Clearly \(S
\cong \Tf(\alpha)\). All of its members are incomparable with the leaf vertex
\(\leaf{\alpha}\), that is, \(S \subset L_{\incomp}(\leaf{\alpha})\). Thus, the
width of \(S \cup \Set{ \leaf{\alpha} }\) as an induced suborder of
\(\Tf(\alpha+1)\) is at least \(\bar\varepsilon(-2 + \alpha) + 1\).

By \cref{prop:tf-width-epsilon}, \(w(\Tf(\alpha + 1))\) is an epsilon number larger
than \(\bar\varepsilon(-2 + \alpha) + 1\). Thus, by definition, \(w(\Tf(\alpha +
1)) \geq \bar\varepsilon(-2 + \alpha + 1)\) and we are done.

Let \(\gamma > 0\) be a limit ordinal such that \(\varepsilon_\gamma \neq \gamma\).
For any \(\alpha < \gamma\), \(\Tf(\alpha) \subset \Tf(\gamma)\) and by the
induction hypothesis \(w(\Tf(\gamma)) \geq \bar\varepsilon(-2 + \alpha)\). Thus,
\(o(\Tf(\gamma)) \geq w(\Tf(\gamma)) \geq
\sup_{\alpha<\gamma}\bar\varepsilon(-2 + \alpha) = \varepsilon_\gamma\). Since
\(\gamma\) is not a fixed point of the enumeration of the epsilon numbers, we have
\(\bar\varepsilon(-2 + \gamma) = \varepsilon_\gamma\) and we are done.

Finally, let \(\gamma > 0\) be a limit ordinal such that \(\varepsilon_\gamma =
\gamma\). Now the previous argument no longer applies, so we take the help of the
qo of finite multisets \(M(\gamma)\) to escape the fixed-point. Our goal
is to embed \(M(\gamma)\) into some small subset of \(\Tf(\gamma)\).

It is a folklore result that \(M(\beta) \cong \Pf(\beta \times \omega)\) for any
well-order \(\beta\) \footnote{This was mentioned to the first author by Harry Altman.}. For concreteness, here is a proof sketch.

\begin{subproof}
Let \(\beta > 0\) be a well-order. Given a multiset \(m \in M(\beta)\), map it to
the set \(s \in \Pf(\beta \times \omega)\) defined as: $$ s = \Set{ (\delta_i,
n_i) \given \delta_i \in m, n_i = |\Set{ \delta' \in m \given \delta' \geq
\delta_i }| - 1}. $$ On the other hand, let \(s \in \Pf(\beta \times \omega)\),
and let \(s' = \Set{ (\delta_i, n_i) }_{i<k}\) be the maximal antichain in \(s\)
such that, for \(i<j<k\), \(\delta_i > \delta_j\) and \(n_i < n_j\). Map it to the
multiset \(m \in M(\beta)\) that has \(n_0+1\) copies of \(\delta_0\), and \((n_i -
n_{i-1})\) copies of \(\delta_i\) for \(0 < i < k\). These two maps witness \(M(\beta)
\cong \Pf(\beta \times \omega)\) in the sense of \cref{def:qo-equivalent}.
\end{subproof}

Back to \(\Tf(\gamma)\). Map an element \((\alpha, k) \in \gamma \times \omega\) to
the tree \(\vertex{\leaf{\alpha}, \leaf{0}^{2+k}} \in \Tf(\gamma)\). Let the range
of this map be \(S \subset \Tf(\gamma)\). Clearly this is an order-embedding, and
\((S,\le_T) \cong \gamma \times \omega\). We can now use
\cref{lemma:finite-powerset-trees} to get \(\bar{S} \cong \Pf(\gamma \times \omega)
\cong M(\gamma)\). Thus, \(o(\bar{S}) = o(M(\gamma)) = \omega^{\widehat\gamma}\) by
\cref{lemma-multiset}, where \(\omega^{\widehat\gamma} = \omega^{\gamma + 1} =
\gamma \cdot \omega> \gamma\). Since \(\gamma\) is an epsilon number and thus an
additively indecomposable ordinal, we can use \cite[Theorem 4.27]{Dzamonja2020}
to also get \(w(\bar{S}) = o(\bar{S}) > \gamma\).

Again, by \cref{prop:tf-width-epsilon}, \(w(\Tf(\gamma))\) is an epsilon number
strictly greater than \(\gamma = \varepsilon_\gamma\), and the least such epsilon
number is \(\varepsilon_{\gamma+1}\). By definition \(\bar\varepsilon(\gamma)
=\varepsilon_{\gamma+1}\) and thus \(w(\Tf(\gamma)) \geq \bar\varepsilon(\gamma)\).

This completes the proof.
\end{proof}

By combining \cref{lemma:Tf1-bound}, \cref{corr:Tf2-lowerbound}, and
\cref{lemma:Tfq-lowerbound}, we get the following proposition.

\begin{proposition}
For any ordinal \(\gamma > 0\), \(o(\Tf(\gamma)) \geq w(\Tf(\gamma)) \geq
g(\gamma)\).
\label{prop:Tfq-lowerbound}
\end{proposition}

\begin{remark}
Specifically for leaf-labelled trees, and with our notation, the lower bound
stated in \cite[Lemma 6.5]{Friedman2023} is the following: $$ o(\Tf(Q \sqcup 1))
\geq \omega^{o(Q)}. $$ This is much weaker than the lower bound proved above:
for example, if \(Q = \delta\) for some ordinal \(\delta \geq \omega\), we get
\(o(\Tf(Q \sqcup 1)) \geq \bar\varepsilon(\delta + 1)\).
\end{remark}
\subsubsection{Results}
\label{subsubsec:results}
We can now finally combine the results of the previous two subsections to give
the precise value of the maximal order type of \(\Tf(Q)\).

\begin{theorem}
Given a non-empty wqo \((Q,\le)\), \(o(\Tf(Q)) = w(\Tf(Q)) = g(o(Q))\), where \(g(1)
= \omega\) and \(g(\alpha) = \bar\varepsilon(-2 + \alpha)\) for \(\alpha \geq 2\).
\label{thm:Tf-order-type}
\end{theorem}

\begin{proof}
We get \(w(\Tf(Q)) \leq o(\Tf(Q)) \leq g(o(Q))\) from \cref{prop:Tfq-upperbound}.

Since \(o(Q) \quasiembedding Q\), we get \(g(o(Q)) \leq w(\Tf(o(Q))) \leq
o(\Tf(o(Q))) \leq o(\Tf(Q))\) from \cref{prop:Tfq-lowerbound}. On the other hand,
any incomparable trees in \(\Tf(o(Q))\) are also necessarily incomparable trees in
\(\Tf(Q)\) (since \(o(Q) \quasiembedding Q\) via a surjective map, modulo
equivalence classes), and so \(w(\Tf(o(Q))) \leq w(\Tf(Q))\) as well.

We get the statement of the theorem by combining both sets of inequalities.
\end{proof}

Unfortunately, the methods used in \cref{subsubsec:lower-bounds} to prove the
lower bound are not immediately formalizable in \(\mathsf{RCA_0}\), with the main
blocker being the proof of \cref{lemma-infinite-width}. It is not too difficult to
modify it for our purposes, but it is a bit tedious and out of scope for this
paper. Instead, let us take a blunter approach.

For a linear order \((\Omega, \leq)\), let \((\varepsilon_\Omega, \preceq)\) be the
linear order consisting of:
\begin{enumerate}
\item constants \(0\), \(\omega\), and \(\varepsilon_\alpha\) for \(\alpha \in \Omega\),
with \(0 \prec \omega \prec \varepsilon_\alpha \preceq \varepsilon_\beta\) for
all \(\alpha, \beta \in \Omega\) with \(\alpha \leq \beta\);
\item terms \(\langle t_0, t_1, \dots, t_{k-1} \rangle\) whenever \(t_0 \succ t_1
   \succ \dots \succ t_{k-1} \in \varepsilon_\Omega\), ordered lexicographically
by \(\preceq\) (and \(t_0 \neq \omega\) and \(t_0 \neq \varepsilon_\alpha\) for any
\(\alpha \in \Omega\) when \(k = 1\));
\item for constant \(c\) and term \(t = \langle t_0, t_1, \dots, t_{k-1} \rangle\), if
\(c \preceq t_0\) then \(c \prec t\), and if \(t_0 \prec c\) then \(t \prec c\).
\end{enumerate}

This is intended to be a notation system for the epsilon numbers with terms in
base-2 normal form. A more precise definition can be found in, for example,
\cite[Definition 3.3]{Freund2023}, which has terms in base-\(\omega\) normal form;
the correspondence between them is straightforward.

Now, let us prove a weaker -- but more useful -- lower bound for \(o(\Tf(Q))\).

\begin{theorem}
In \(\mathsf{RCA_0}\), \(\varepsilon_\Omega \quasiembedding \Tf(2 \cup \Omega)\)
(where \(0 < 1 < \alpha\) for \(\alpha \in \Omega\) in \(2 \cup \Omega\)).
\label{thm:map-epsilon-to-Tf}
\end{theorem}

\begin{proof}[Proof sketch]
Let any term \(t \in \varepsilon_\Omega\) such that \(t \prec \omega\) be called
\emph{finite} (including the constant \(0\)). We can map finite terms to their \emph{value}
\(v(t)\) in \(\mathbb{N}\).

We inductively define a map \(h \colon \varepsilon_\Omega \to \Tf(2 \cup \Omega)\)
as follows.
\begin{enumerate}
\item For a finite term \(t\) with \(v(t) = n\), \(h(t) = \vertex{\leaf{1},
   \leaf{0}^{n+2}}\).
\item \(h(\omega) = \leaf{1}^3\) and \(h(\varepsilon_\alpha) = \leaf{\alpha}\) for
\(\alpha \in \Omega\).
\item For a term \(t = \langle t_0, t_1, \dots, t_{k-1}, \dots, t_{k+l-1} \rangle\)
such that \(t_f = \langle t_k, \dots, t_{k+l-1} \rangle\) is finite with
\(v(t_f) = n\), let \(\tau_i = \vertex{h(t_i), \leaf{0}^{m_i}}\) with \(m_i =
   \sum_{j=0}^{i} (ht(h(t_j)) + 2)\) for every \(i < k\), and \(\tau_f =
   \vertex{h(t_f), \leaf{0}^{m_f}}\) with \(m_f = m_k + n + 4\). Then \(h(t) =
   \vertex{\tau_0, \tau_1, \dots, \tau_{k-1}, \tau_f}\). (If \(l = 0\), ignore
\(t_f\) and \(\tau_f\).)
\end{enumerate}

The map \(h\) is order-reflecting; this can be readily verified by case analysis.
\end{proof}

We can now use this to derive some reverse mathematical corollaries. Recall that
the proof-theoretic ordinal of an arithmetical theory is, informally, the
supremum of the order types of computably definable linear orders that are
provably well-orders. Simply by using two comparable labels, we get a statement
which is independent of \(\mathsf{ACA_0}\), whose proof-theoretic ordinal is
well-known to be \(\varepsilon_0\). This is a remarkable jump in arithmetical
strength, given that \(o(\Tf(1))\) is only \(\omega\), and we do not even require
any incomparable elements.

\begin{corollary}
The statement ``\(\Tf(2)\) is a wqo'' is not provable in \(\mathsf{ACA}_0\).
\end{corollary}

For linear orders \(\Omega\) that are infinite and have a definable successor
function, we have \(2 \cup \Omega \quasiembedding \Omega\). Thus, we can get
similar independence results for the theories \(\mathsf{ACA}'\) and
\(\mathsf{ACA}\), whose proof-theoretic ordinals are \(\varepsilon_\omega\) and
\(\varepsilon_{\varepsilon_0}\) respectively, using appropriate wqos.

There has been a line of research studying the reverse-mathematical strength of
so-called \emph{well-ordering principles}: statements of the form ``if \(\Omega\) is a
well-order, then so is \(\mathsf{F}(\Omega)\)'', where \(\mathsf{F}\) is some
specific operator on linear orders. The following theorem was proved in
\cite{Marcone2020} as well as \cite{Afshari2009}, using
computability-theoretic and proof-theoretic methods respectively.

\begin{theorem}
Over \(\mathsf{RCA}_0\), the statement ``if \(\Omega\) is a well-order, then so is
\(\varepsilon_{\Omega}\)'' is equivalent to \(\mathsf{ACA}^+_0\).
\label{thm:wop-aca0p}
\end{theorem}

This leads to the following equivalence, barring some details.

\begin{theorem}
Over \(\mathsf{RCA_0}\), the statement ``if \((Q,\le)\) is a wqo, then
\(\Tf(Q)\) is also a wqo'' is equivalent to \(\mathsf{ACA^+_0}\).
\end{theorem}

\begin{proof}[Proof sketch]
The forward direction follows from \cref{thm:map-epsilon-to-Tf} and
\cref{thm:wop-aca0p}.

It remains to show that the statement ``\((Q,\le)\) is a wqo implies \(\Tf(Q)\) is a
wqo'' can be proven in \(\mathsf{ACA^+_0}\). The classical approach to attain this
sort of formalization would be via the notion of \emph{reifications} (see \cite[Section 4]{Simpson1988}). Given a reification of \(Q\) by \(\alpha\), it is not too difficult to
use the techniques in the proof of \cref{prop:Tfq-upperbound} to show a
reification of \(\Tf(Q)\) by \(\varepsilon_{\alpha+1}\). We instead sketch a
different proof.

We reason in \(\mathsf{ACA_0^+}\), and assume that \(Q\) is a wqo while \(\Tf(Q)\) is
not. It is known that \(\mathsf{ACA^+_0}\) proves, for a finite number of
arbitrary sets, the existence of a countable coded \(\omega\)-model of
\(\mathsf{ACA_0}\) containing those sets (e.g., see \cite[Lemma
3.4]{Afshari2009}). Let \(\mathcal{M}\) be an \(\omega\)-model of \(\mathsf{ACA_0}\) containing \(Q\)
and an infinite bad sequence of \(\Tf(Q)\). Our goal is to show that \(\mathcal{M}\)
thinks that in fact \(\Tf(Q)\) is a wqo, and thus reach a contradiction.

For this, we consider a family of suborders of \(Q\) generated by bad sequences of
\(Q\): containing, for a bad sequence \(\sigma = \langle q_0, q_1, \dots, q_{k-1}
\rangle\), the set $$ L(\sigma) = \Set{q \in Q \given \forall i < k. \, \, q_i
\not\le q}, $$ with the family ordered by inclusion. This is clearly definable
in \(\mathcal{M}\), and since \(Q\) is a wqo, the inclusion order is well-founded.
Note that \(L(\langle\rangle) = Q\).

From the analysis in \cref{subsubsec:upper-bounds}, for any \(\tau \in
\Tf(L(\sigma))\), \(L(\tau)\) being a wqo is implied by \(\Tf(L(\sigma + \langle q_0
\rangle)), \dots, \Tf(L(\sigma + \langle q_k \rangle))\) being wqos (where
\(q_0,\dots,q_k\) are the labels of the leaf vertices of \(\tau\)), along with
closure properties of \(\Pf\) and \(\sqcup\); this holds in \(\mathsf{ACA_0}\). Since
assertions that \(\Tf(L(\sigma))\) is a wqo in \(\mathcal{M}\), for any \(\sigma\),
are arithmetical, we can prove that all \(\Tf(L(\sigma))\) are wqos according to
\(\mathcal{M}\) (and hence also \(\Tf(Q)\)) by an instance of transfinite induction
that is formalizable in \(\mathsf{ACA^+_0}\).
\end{proof}
\section{Transfinite Sequences with Finite Range}
\label{sec:transfinite-sequences}
We now turn our attention to qos on transfinite sequences ordered by
embeddability, as a generalization of the Higman order on finite sequences (or
strings). As we know, in order to ensure that this qo is also a wqo when the
underlying qo of elements is a wqo, we need to restrict the collection of
sequences to ones with only \emph{finite range}. Specifically, we consider the qo of
\emph{indecomposable sequences} with finite range. The main result in this section is
that the qo of all indecomposable sequences with finite range of length less
than \(\omega^\omega\) is essentially the same as the qo on leaf-labelled trees
ordered by homomorphisms discussed above. This lets us reuse the values of the
maximal order types for the qo on indecomposable sequences, and eventually for
\emph{arbitrary} transfinite sequences with finite range of length less than
\(\omega^\omega\). This serves as partial progress towards a long-standing open
problem in this field, which we'll talk about in the next subsection.
\subsection{Transfinite sequences}
\label{subsec:transfinite-sequences}
We first introduce some basic definitions and facts about transfinite sequences.

\begin{definition}
A \textbf{transfinite sequence} of \textbf{length} (or \textbf{rank}) \(\alpha\) over a qo \((Q,\le)\) is
a function \(\sigma \colon \alpha \to Q\), where \(\alpha\) is a non-zero ordinal.
The length of \(\sigma\) is also denoted by \(|\sigma|\).

The collection of \emph{all} transfinite sequences over \(Q\) is denoted by
\(\mathbf{s(Q)}\), while those with length strictly less than some non-zero
ordinal \(\alpha\) will be denoted by \(s_\alpha(Q)\). Finite sequences over \(Q\)
(i.e., sequences in \(s_\omega(Q)\)) will usually be represented as \(\langle q_0,
q_1, \dots, q_{k-1} \rangle\) for \(q_i \in Q\) and \(i < k < \omega\).
\end{definition}

\begin{remark}
The empty sequence -- a sequence of length 0 -- is deliberately excluded from
the definition above, and will \emph{not} be an element of the qos defined in this
section. As will become clear later, this decision is analogous to excluding the
empty tree from \(\Tf(Q)\) in the previous section.
\end{remark}

We can define a few simple operations on transfinite sequences -- namely those
of \emph{concatenation} and \emph{exponentiation} -- which we shall be using in the
upcoming sections.

\begin{definition}
Given sequences \(\sigma\) and \(\tau\) over \(Q\), their \textbf{concatenation} \(\sigma +
\tau\) is a transfinite sequence of length \(|\sigma| + |\tau|\) such that
$$ (\sigma + \tau)(\iota) =
\begin{cases}
\sigma(\iota), & \iota < |\sigma|\\
\tau(-|\sigma| + \iota), & |\sigma| \leq \iota < |\sigma| + |\tau|
\end{cases} $$
\end{definition}

Essentially, \(\sigma + \tau\) is the sequence obtained by placing a copy of
\(\tau\) after a copy of \(\sigma\).

\begin{definition}
Given a sequence \(\sigma\) over \(Q\), the sequence \(\sigma^\alpha\) denotes the
\textbf{exponentiation} of \(\sigma\) to the (non-zero) power \(\alpha\). The sequence
\(\sigma^\alpha\) has length \(|\sigma| \cdot \alpha\) and, for any \(\beta
< |\sigma| \cdot \alpha\) such that \(\beta = |\sigma| \cdot \delta + \iota\) for
\(\delta < \alpha\) and \(\iota < |\sigma|\), \(\sigma^\alpha(\beta) =
\sigma(\iota)\).
\end{definition}

Essentially, \(\sigma^\alpha\) is the sequence obtained by concatenating
``\(\alpha\) many'' copies of \(\sigma\). This is more concrete when \(\alpha\) is
some finite \(k\): \(\sigma^k\) is simply \(\sigma + \sigma + \dots + \sigma\) (\(k\)
many times).

In order to form a qo from these collections of transfinite sequences,
we now define a way of comparing them by way of \emph{sequence embeddings}.

\begin{definition}
Consider a qo \((Q,\le_Q)\) and sequences \(\sigma, \tau\) over \(Q\) of length
\(\alpha, \beta\) respectively. A \textbf{sequence embedding} from \(\sigma\) to \(\tau\) is
a map \(f \colon \alpha \to \beta\) such that:
\begin{enumerate}
\item \(f\) is strictly monotone: if \(i < j < \alpha\) then \(f(i) < f(j) < \beta\);
\item \(f\) respects the ordering on the elements: \(\sigma(i) \le_Q \tau(f(i))\) for
\(i < \alpha\).
\end{enumerate}

If such an embedding exists, \(\sigma\) is \textbf{embeddable} in \(\tau\). The qo
\(\preceq\) on transfinite sequences is defined as: \(\sigma \preceq \tau\) iff \(\sigma\)
is embeddable in \(\tau\).
\end{definition}

The qo \(\preceq\) respects lengths of sequences: if \(\sigma \preceq \tau\) then
\(|\sigma| \leq |\tau|\); or equivalently, if \(|\sigma| \not\leq |\tau|\) then
\(\sigma \not\preceq \tau\).

This lets us define qos \((s(Q),\preceq)\) and \((s_\alpha(Q), \preceq)\) (for
any non-zero ordinal \(\alpha\)) of transfinite sequences over a qo \(Q\),
of arbitrary length and of length strictly less than \(\alpha\), respectively,
ordered by sequence embeddability. As usual, we will identify both the
qos and the underlying sets by \(s(Q)\) and \(s_\alpha(Q)\) when there is
no ambiguity.

The restriction to \(\alpha = \omega\) is quite well known.

\begin{examples}
The qo \((s_\omega(Q), \preceq)\) is exactly the set of finite sequences
over \(Q\). It is also called the \emph{Higman ordering} over \(Q\), which is denoted by
\((Q^*, \le_*)\).
\end{examples}

It is known that the Higman ordering over \(Q\) is a wqo whenever \(Q\)
itself is a wqo, with a well-defined maximal order type based on
the maximal order type of \(Q\) \cite{Higman1952,deJongh1977}. Thus, it is
natural to ask whether the same holds for any \(\alpha\). Unfortunately, Rado
showed that things fall apart at the very first step \cite[Theorem 2]{Rado1954}.

\begin{proposition}
There is a wqo \(R\) such that \(s_{\omega+1}(R)\) is not a
wqo. The wqo \(R\) is often referred to as \emph{Rado's order} or \emph{Rado's
counter-example}.
\end{proposition}

Fortunately, there is a way to salvage the situation by focusing on transfinite
sequences of \emph{finite range}.

\begin{definition}
The \textbf{range} of a sequence \(\sigma \in s(Q)\) is the set of elements from \(Q\) that
appear in it: \(\text{range}(\sigma) = \Set{x \in Q \given \exists i < |\sigma|.
\, \, \sigma(i) = x}\).

If the range of a sequence \(\sigma\) is finite, that is, \(|\text{range}(\sigma)|
< \omega\), then the sequence is said to have \textbf{finite range} (or, is \textbf{finitary}).
\end{definition}

\begin{definition}
Given a qo \((Q,\le)\) we define the qo \((s^F_\alpha(Q), \preceq)\)
(with \(\alpha\) any non-zero ordinal), where \(\mathbf{s^F_\alpha(Q)}\) is the
collection of all transfinite sequences over \(Q\) \emph{with finite range} of length
less than \(\alpha\).
\end{definition}

This leads us to the required generalization of the result for the Higman
ordering.

\begin{theorem}[Nash-Williams' theorem]
If \((Q,\le)\) is a wqo, then \((s^F_\alpha(Q),\preceq)\) is also a
wqo for any ordinal \(\alpha\).
\label{thm:nash-williams}
\end{theorem}

This theorem was proved by Nash-Williams \cite{NashWilliams1965} as a broad
generalization of existing results for specific \(\alpha\). The proof assumes the
existence of a bad sequence of elements of \(s^F_\alpha(Q)\), showing that it
leads to a contradiction. The tools that are used -- in particular the so-called
\emph{thin sets} -- are the precursors of what would become the theory of
\emph{better-quasi-orders}. Of course, as mentioned before, there are no known
methods to extract information about the maximal order type of \(s^F_\alpha(Q)\)
from this proof.

A more constructive proof was given earlier by Erdős and Rado \cite{Erdos1959}
for ordinals less than \(\omega^\omega\).

\begin{proposition}
If \((Q,\le)\) is a wqo, then for any \(\alpha < \omega^\omega\),
\((s^F_\alpha(Q),\preceq)\) is also a wqo.
\end{proposition}

Schmidt \cite[Section 3.3]{DianaSchmidt2020} proposed non-trivial upper bounds
for the maximal order type of \(s^F_\alpha(Q)\) for any \(\alpha\), but her proof
has a fatal flaw \footnote{This was communicated to the first author by Andreas Weiermann, who
seems to be the first person to catch the error.}. Specifically, in the proof of Theorem 3.6 for
case \(\alpha \in \{\beta, \beta+1 \}\), the sequence \(g\) that is constructed may
no longer have finite range.

In the remainder of this section, we consider the case for \(\alpha =
\omega^\omega\), that is, the qo \(s^F_{\omega^\omega}(Q)\), as a generalization of
the result by Erdős and Rado, and calculate its maximal order type as a function
of the maximal order type of \(Q\). As mentioned above, this serves as partial
progress towards calculating \(o(s^F_\alpha(Q))\) for any \(\alpha\), left open
after Schmidt's erroneous proof.
\subsection{Indecomposable transfinite sequences}
\label{subsec:indecomposable-sequences}
In order to study arbitrary transfinite sequences, we first consider a
particularly uniform subset of \(s(Q)\).

\begin{definition}
A \textbf{tail} of a transfinite sequence \(\sigma \in s_\alpha(Q)\) is a non-empty
sequence \(\tau\) such that for a fixed \(\delta < |\sigma|\), \(\tau(i) =
\sigma(\delta + i)\) for \(0 \leq i < -\delta + |\sigma|\). If \(\delta > 0\) it is
called a \textbf{proper tail} starting at index \(\delta\).
\end{definition}

\begin{definition}
A transfinite sequence \(\sigma\) is \textbf{indecomposable} if it is embeddable into
every tail of itself: if \(\tau\) is any tail then \(\sigma \preceq \tau\). In fact,
since \(\tau \preceq \sigma\) trivially, \(\sigma\) and \(\tau\) are equivalent (mod
\(\preceq\)).

The collection of all indecomposable transfinite sequences over \(Q\) is denoted
by \(\mathbf{i(Q)}\).
\end{definition}

It readily follows that the length of any indecomposable sequence must be an
\emph{additively indecomposable ordinal} of the form \(\omega^\alpha\) for some \(\alpha\).
Note that sequences of the form \(\sigma^\omega\) are always indecomposable, with
length \(|\sigma| \cdot \omega\).

Indecomposable sequences are very well behaved when it comes to embeddability in
yet another sense:

\begin{proposition}
Given an indecomposable sequence \(\sigma \in i(Q)\) and arbitrary sequences
\(\tau_0, \tau_1 \in s(Q)\), if \(\sigma \preceq \tau_0 + \tau_1\) then either \(\sigma
\preceq \tau_0\) or \(\sigma \preceq \tau_1\).
\end{proposition}

\begin{proof}
Let \(f \colon |\sigma| \to |\tau_0| + |\tau_1|\) be an embedding. If the range of
\(f\) is contained in \(|\tau_0|\), then \(\sigma\) is completely embedded in
\(\tau_0\). Otherwise, assume there is a tail of \(\sigma\) starting at some index
\(\delta < |\sigma|\), say \(\sigma_\delta\), such that \(f\) completely embeds
\(\sigma_\delta\) into \(\tau_1\). Since \(\sigma\) is indecomposable, there is an
embedding \(g\) of \(\sigma\) into \(\sigma_\delta\). The composition \(f \circ g\)
gives an embedding of \(\sigma\) into \(\tau_1\) and we are done.
\end{proof}

This proposition can be generalized in a straightforward manner.

\begin{proposition}
Given an indecomposable sequence \(\sigma \in i(Q)\) and a finite collection of
arbitrary sequences \(\tau_i \in s(Q)\) for \(i < k < \omega\), if \(\sigma \preceq
\tau_0 + \tau_1 + \dots + \tau_{k-1}\), then there is a \(j<k\) such that \(\sigma
\preceq \tau_j\).
\label{prop-atomic-indecomposable}
\end{proposition}

In order to study the qos \(s^F_\alpha(Q)\) we define the following
induced suborder consisting only of indecomposable sequences.

\begin{definition}
Given a qo \((Q,\le)\) we define the qo \((i^F_\alpha(Q), \preceq)\)
(with \(\alpha\) any non-zero ordinal), where \(\mathbf{i^F_\alpha(Q)}\) is the
collection of all indecomposable transfinite sequences over \(Q\) with \emph{finite
range} of length less than \(\alpha\).
\end{definition}

Since \(i^F_\alpha(Q)\) is an induced suborder of \(s^F_\alpha(Q)\), by
\nameref{thm:nash-williams}, \(i^F_{\omega^\omega}(Q)\) is a wqo whenever \(Q\) is a
wqo. Our goal, in the following sections, is to compute the maximal order type
of \(i^F_{\omega^\omega}(Q)\) in terms of the maximal order type of \(Q\).
\subsection{Correspondence with finite leaf-labelled trees}
\label{subsec:tree-sequence-correspondence}
In this section we show that the qos \(i^F_{\omega^\omega}(Q)\) and
\(\Tf(Q)\) are equivalent (in the sense of \cref{def:qo-equivalent}), which would
let us simply reuse the results for \(o(\Tf(Q))\) from
\cref{subsec:finite-leaf-labelled-trees-type}.

The route from leaf-labelled trees to indecomposable sequences is relatively
straightforward.

\begin{definition}
For any qo \(Q\), we define a map \(f \colon \Tf(Q) \to
i^F_{\omega^\omega}(Q)\) recursively based on the structure of a tree \(\tau \in
\Tf(Q)\) as follows:
\begin{enumerate}
\item If \(\tau = \leaf{q}\) for \(q \in Q\), then \(f(\tau) = \langle q \rangle\) with
\(|f(\tau)| = 1\).
\item If \(\tau = \vertex{\tau_0, \tau_1, \dots, \tau_{k-1}}\), then \(f(\tau) =
   (f(\tau_0) + f(\tau_1) + \dots + f(\tau_{k-1}))^\omega\), with \(|f(\tau)| =
   (\max_{i<k}|f(\tau_i)|) \cdot \omega\).
\end{enumerate}
\end{definition}

As mentioned before, any sequence of the form \(\sigma^\omega\) is an
indecomposable sequence. By a simple induction, if \(ht(\tau) = n\) then
\(|f(\tau)| = \omega^n < \omega^\omega\). Additionally, the range of \(f(\tau)\) is
the set of labels on the leaf vertices of \(\tau\) (and hence finite). Thus, the
map \(f\) is well-defined.

\begin{lemma}
For any qo \((Q,\le_Q)\), the map \(f \colon \Tf(Q) \to
i^F_{\omega^\omega}(Q)\) is an order-embedding: for any \(\sigma, \tau \in
\Tf(Q)\), \(\sigma \le_T \tau\) iff \(f(\sigma) \preceq f(\tau)\).
\label{lemma:f-order-embedding}
\end{lemma}

\begin{proof}
Consider \(\sigma, \tau \in \Tf(Q)\). We induct on the structure of \(\sigma\) and
\(\tau\).

When both \(\sigma = \leaf{p}\) and \(\tau = \leaf{q}\) are leaf vertices, for \(p,q
\in Q\), then by definition \(\leaf{p} \le_T \leaf{q}\) iff \(p \le_Q q\) iff \(\langle p
\rangle \preceq \langle q \rangle\).

Let \(\sigma = \leaf{p}\) for \(p \in Q\), and \(\tau = \vertex{\tau_0, \tau_1,
\dots, \tau_{l-1}}\). By definition, \(\sigma \le_T \tau\) iff there is some \(j<l\)
such that \(\sigma \le_T \tau_j\), and by the induction hypothesis it follows that
\(f(\sigma) \preceq f(\tau_j) \preceq f(\tau)\). Conversely, \(f(\sigma) \preceq f(\tau)\)
also implies that there is some \(j<l\) such that \(f(\sigma) \preceq f(\tau_j)\)
(since \(f(\sigma) = \langle p \rangle\) and \(|\langle p \rangle| = 1\)), and by
the induction hypothesis it follows that \(\sigma \le_T \tau_j \le_T \tau\).

Now let \(\sigma = \vertex{\sigma_0, \sigma_1, \dots, \sigma_{k-1}}\). If \(\tau =
\leaf{q}\) is a leaf vertex, for \(q \in Q\), then clearly \(\sigma \not\le_T \tau\)
and \(f(\sigma) \not\preceq f(\tau) = \langle q \rangle\) (since \(|f(\sigma)| > 1\)).
So let \(\tau = \vertex{\tau_0, \tau_1, \dots, \tau_{l-1}}\).

Assume \(\sigma \le_T \tau\). By definition, for all \(i < k\) there is a \(j < l\) such
that \(\sigma_i \le_T \tau_j\), and by induction hypothesis we have \(f(\sigma_i)
\preceq f(\tau_j)\). Thus \((\sigma_0 + \sigma_1 + \dots + \sigma_{k-1}) \preceq
(\tau_0 + \tau_1 + \dots + \tau_{l-1})^k\), by embedding \(\sigma_i\) into \(\tau_j\)
in the \(i^\text{th}\) copy of \((\tau_0 + \tau_1 + \dots + \tau_{l-1})\). This
extends to an embedding of \(f(\sigma) = (\sigma_0 + \sigma_1 + \dots +
\sigma_{k-1})^\omega\) into \(((\tau_0 + \tau_1 + \dots + \tau_{l-1})^k)^\omega =
(\tau_0 + \tau_1 + \dots + \tau_{l-1})^\omega = f(\tau)\), and thus \(f(\sigma)
\preceq f(\tau)\).

Conversely, let \(f(\sigma) \preceq f(\tau)\). Since \((f(\sigma_0) + f(\sigma_1) +
\dots + f(\sigma_{k-1}))\) must be embeddable in a finite prefix of \(f(\tau)\),
for every \(i<k\) there must be an \(m_i < \omega\) such that \(f(\sigma_i) \preceq
(f(\tau_0) + f(\tau_1) + \dots + f(\tau_{l-1}))^{m_i}\). By
\cref{prop-atomic-indecomposable}, for every \(i<k\) there must be a \(j<l\) such that
\(f(\sigma_i) \preceq f(\tau_j)\), and by the induction hypothesis we get
\(\sigma_i \le_T \tau_j\). Thus, by definition, \(\sigma \le_T \tau\).
\end{proof}

The route from indecomposable sequences to leaf-labelled trees is a bit more
hairy. First, we need to define a way to further decompose indecomposable
sequences.

\begin{definition}
For an indecomposable sequence \(\sigma \in i^F_\alpha(Q)\), the set of its
\emph{cofinal factors} is denoted by \(\factors(\sigma)\) and is defined as: $$
\factors(\sigma) = \Set{ \tau \in i^F_{|\sigma|}(\text{range}(\sigma)) \given
\tau^\omega \preceq \sigma }. $$
\end{definition}

\begin{definition}
For any qo \(Q\), we define a (partial) map \(g \colon
i^F_{\omega^\omega}(Q) \to \Tf(Q)\) recursively based on the length of a sequence
\(\sigma \in i^F_{\omega^\omega}(Q)\) as follows:
\begin{enumerate}
\item If \(\sigma = \langle q \rangle\) for \(q \in Q\), then \(g(\sigma) = \leaf{q}\)
with \(ht(g(\sigma)) = 0\).
\item Otherwise, if \(|\sigma| = \omega^n\) for \(0 < n < \omega\), and
\(\factors(\sigma) = \Set{ \sigma_0, \sigma_1, \dots, \sigma_{k-1} }\) is a
finite set, then \(g(\sigma) = \vertex{g(\sigma_0), g(\sigma_1), \dots,
   g(\sigma_{k-1})}\) with \(ht(g(\sigma)) = n\).
\end{enumerate}
\end{definition}

Unfortunately, while it is clear that the range of \(g\) consists of finite
leaf-labelled trees, it is not clear that it is a total function: \(g(\sigma)\) is
not defined if \(\factors(\sigma)\) is not a finite set. Our goal is to show that
this cannot happen for indecomposable sequences with finite range.

For analysing the embedding \(g\), it will be helpful to have some notation for
sequences of an exact specified length.

\begin{definition}
For a qo \(Q\) and ordinal \(\alpha\), \(s_{=\alpha}^F(Q)\) and
\(i_{=\alpha}^F(Q)\) are subsets (or induced suborders) of \(s_{\alpha+1}^F(Q)\) and
\(i_{\alpha+1}^F(Q)\), respectively, containing transfinite sequences of length
\emph{exactly} \(\alpha\).
\end{definition}

\begin{proposition}
Fix an \(n < \omega\), and assume that for every finite qo \(P\) (i.e.,
\(|P| < \omega\)) \(i_{\omega^n}^F(P)\) is a wqo. Then the following
holds for any qo \(Q\):
\begin{enumerate}
\item Every \(\sigma \in s_{=\omega^n}^F(Q)\) is either indecomposable (i.e., \(\sigma
   \in i_{=\omega^n}^F(Q)\)), or can be decomposed as the concatenation
\(\sigma_0 + \sigma_1\) of a sequence \(\sigma_0 \in s_{\omega^n}^F(Q)\) and an
indecomposable sequence \(\sigma_1 \in i_{=\omega^n}^F(Q)\).
\item Every \(\sigma \in s_{\omega^{n+1}}^F(Q)\) can be decomposed as a finite
concatenation \(\sigma_0 + \sigma_1 + \cdots + \sigma_{k-1}\) of indecomposable
sequences \(\sigma_i \in i_{\omega^{n+1}}^F(Q)\) for all \(i < k\).
\end{enumerate}
\label{prop:sequence-decomposition}
\end{proposition}

\begin{proof}
We induct on \(n\). For \(n = 0\), \(s^F_{=1}(Q) = i^F_{=1}(Q) \cong Q\), and
\(s^F_\omega(Q) = Q^*\) consists of finite sequences, thus both statements are
trivial.

For any \(n > 0\), statement (2) holds by a simple induction on the length of an
arbitrary \(\sigma \in s_{\omega^{n+1}}^F(Q)\). If \(|\sigma| = \omega^m\) for \(m <
n+1\) (i.e., \(\sigma \in s_{=\omega^m}^F(Q)\)), then by combining the induction
hypothesis for statements (1) and (2), \(\sigma\) can be decomposed as \(\sigma_0 +
\sigma_1 + \cdots + \sigma_{k-1} + \sigma_k\), where \(\sigma_i \in
i^F_{\omega^m}(Q)\) for \(i < k\) and \(\sigma_k \in i^F_{=\omega^m}(Q)\). Otherwise,
if \(|\sigma| = \alpha + \omega^m < \omega^{n+1}\), where \(\omega^m\) is the
smallest ordinal in the Cantor normal form of \(|\sigma|\), we induct on the
prefix of length \(\alpha\) and apply the previous reasoning to the tail of length
\(\omega^m\).

Now we focus on statement (1). Let \(n = 1\). Consider a sequence \(\sigma \in
s_{=\omega}^F(Q)\) of length \(\omega\), and define \(P \coloneq
\text{range}(\sigma)\). By assumption, \(i_\omega^F(P) \cong P\) is a wqo (and also
trivially, since \(P\) is finite). Define \(P_\mathit{cf} \subset P\) to be the
elements of \(P\) that occur cofinally in \(\sigma\). More precisely, $$
P_\mathit{cf} \coloneq \Set{ p \in P \given \forall i.\, \exists j > i.\,\,
\sigma(j) = p }. $$ \(P_\mathit{cf}\) cannot be empty, since otherwise every \(p
\in P = \text{range}(\sigma)\) would have a largest index where it occurs; since
\(P\) is finite, this would contradict \(|\sigma| = \omega\). If \(P =
P_\mathit{cf}\), then clearly \(\sigma \in i_{=\omega}^F(Q)\). Assume otherwise,
and let \(m \coloneq 1 + \max_{i < \omega} \sigma(i) \in P \setminus
P_\mathit{cf}\).

Let \(\sigma_0 \coloneq \langle \sigma(0), \sigma(1), \dots, \sigma(m-1) \rangle\)
be a proper prefix of \(\sigma\) of length \(m\), and \(\sigma_1\) be the proper tail
of \(\sigma\) starting at index \(m\). By definition, \(\sigma_0 \in
s_{\omega}^F(Q)\), and since all elements that occur in \(\sigma\) after the index
\(m\) occur cofinally, \(\sigma_1 \in i_{=\omega}^F(Q)\). This completes the
decomposition.

Now we generalize this proof. Let \(n > 1\). Consider a sequence \(\sigma \in
s_{=\omega^n}^F(Q)\), and define \(P \coloneq \text{range}(\sigma)\). We can
consider \(\sigma\) as a sequence \(\sigma_0 + \sigma_1 + \dots\) of length \(\omega\)
of elements of \(s_{=\omega^{n-1}}^F(P)\). (Note that this may not be representable
as an element of \(s_{=\omega}^F(s_{=\omega^{n-1}}^F(P))\) since the range is not
necessarily finite!)

By induction hypothesis, each \(\sigma_i\) can be decomposed as \(\sigma_{i0} +
\cdots + \sigma_{i{(k-1)}} + \sigma_{ik}\) of elements of \(i_{\omega^{n-1}}^F(P)\)
and \(i_{=\omega^{n-1}}^F(P)\), and thus as elements of \(i_{\omega^n}^F(Q)\). So,
consider \(\sigma = \tau_0 + \tau_1 + \cdots\) as a sequence \((\tau_i)_{i \in
\omega}\) of length \(\omega\) of elements of \(i_{\omega^n}^F(P)\).

Since \(P\) is finite, \(i_{\omega^n}^F(P)\) is a wqo by assumption. Define a
certain set of indices \(I\) as: $$ I \coloneq \Set{ i < \omega \given \forall j >
i. \,\, \tau_i \not\preceq \tau_j }. $$ Clearly, for indices \(i_0, i_1 \in I\)
with \(i_0 < i_1\), we have \(\tau_{i_0} \not\preceq \tau_{i_1}\). If \(|I| =
\omega\), we would have a bad sequence \((\tau_i)_{i \in I}\) of elements of
\(i_{\omega^n}^F(P)\), which we assumed to be a wqo. Thus, \(|I| < \omega\).

Let \(m \coloneq 1 + \max I\) if \(I\) is non-empty, and \(m \coloneq 0\) otherwise.
By the definition of \(I\), for any \(i \geq m\) there exists \(j > i\) such that
\(\tau_i \preceq \tau_j\). This implies that we can make the index \(j\) arbitrarily
large: for any \(l \geq 0\), there exists such a \(j > i + l\); since we can use the
property to get an ascending sequence \(\tau_i = \tau_{i_0} \preceq \tau_{i_1}
\preceq \cdots \tau_{i_l} \preceq \tau_{i_{l+1}}\) with \(i_{l+1} > i + l\). Thus
\(\langle \tau_{m}, \tau_{m+1}, \dots \rangle\) is an indecomposable sequence of
elements of \(i_{\omega^n}^F(P)\).

Let \(\sigma_0 \coloneq \tau_0 + \tau_1 + \cdots + \tau_{m-1}\) be a prefix of
\(\sigma\) if \(m \neq 0\), and let \(\sigma_1 \coloneq \tau_{m} + \tau_{m+1} +
\cdots\) be the tail of \(\sigma\) starting at index \(m\) (as a sequence of elements
of \(i^F_{\omega^n}(Q)\)). If \(m = 0\) then \(\sigma = \sigma_1\) and \(\sigma \in
i^F_{=\omega^n}(Q)\). Otherwise, \(\sigma_0 \in s_{\omega^n}^F(Q)\), and since
\(\omega^n\) is an additively indecomposable ordinal, \(|\sigma_1| = \omega^n\) and
\(\sigma_1 \in i^F_{=\omega^n}(Q)\). This completes the proof.
\end{proof}

\begin{note}
Strictly speaking, the assumption that \(i_{\omega^n}^F(P)\) is a wqo for all
finite \(P\) is unnecessary: the upcoming \cref{lemma:tfn-finiteness} and
\cref{prop:g-order-embedding} will imply that \(i_{\omega^n}^F(P)\) is simply finite
if \(P\) is finite. It also follows from the results in \cite{Erdos1959}. Still,
we have chosen to keep \cref{prop:sequence-decomposition} in a slightly more
complicated form in order to not have repetitive reasoning.
\end{note}

Just like before, it will be helpful to have some notation for trees of bounded
and fixed heights, as counterparts to \(i^F_{\omega^n}(Q)\) and
\(i^F_{=\omega^n}(Q)\).

\begin{definition}
For a qo \(Q\) and \(n < \omega\), \(\Tf^{<n}(Q)\) and \(\Tf^{=n}(Q)\) are
subsets (or induced suborders) of \(\Tf(Q)\) consisting of trees of height less
than \(n\), and equal to \(n\) respectively.
\end{definition}

\begin{lemma}
For any \(n < \omega\) and finite qo \(P\), the qos \(\Tf^{<n}(P)\)
and \(\Tf^{=n}(P)\) are finite (modulo equivalence).
\label{lemma:tfn-finiteness}
\end{lemma}

\begin{proof}
We induct on \(n\), for a fixed finite qo \(P\). Since \(\Tf^{=n}(P)\) is an induced
suborder of \(\Tf^{<n+1}(P)\), we focus on the statement for \(\Tf^{<n}(P)\).

For \(n = 0\) the qo is empty, and for \(n = 1\) we have \(\Tf^{<1}(P) \cong P\) and
hence finite by assumption.

Let \(n \geq 2\). In the notation of \cref{lemma:finite-powerset-trees},
\(\Tf^{<n}(P) = (L \cup \overline{\Tf^{<{n-1}}(P)}, \le_T)\) where \(L \coloneq
\Set{ \leaf{p} \given p \in P }\). Since \(\overline{\Tf^{<{n-1}}(P)} \cong
\Pf(\Tf^{<{n-1}}(P))\), \(|\overline{\Tf^{<{n-1}}(P)}| \leq 2^{|\Tf^{<{n-1}}(P)|}\)
by the upper bound in \cref{lemma:finite-powerset-type}. By the induction
hypothesis, \(\Tf^{<{n-1}}(P)\) is finite, and thus \(|\Tf^{<n}(P)|
= |L| + |\overline{\Tf^{<{n-1}}(P)}| \leq |P| + 2^{|\Tf^{<{n-1}}(P)|}\) is also
finite.
\end{proof}

\begin{proposition}
\(g \colon i_{\omega^\omega}^F(Q) \to \Tf(Q)\) is a well-defined order-embedding.
\label{prop:g-order-embedding}
\end{proposition}

\begin{proof}
We show the following by induction on \(n\): for every \(n < \omega\), \(g
\upharpoonright i_{\omega^{n+1}}^F(Q)\) is a well-defined order-embedding into
\(\Tf^{<n+1}(Q)\). The general statement of the proposition follows.

For \(n = 0\) the statement holds by definition.

Let \(n > 0\). By induction hypothesis, \(g \upharpoonright i_{\omega^n}^F(Q)\) is a
well-defined order-embedding into \(\Tf^{<n}(Q)\).

Consider any \(\sigma \in i_{=\omega^n}^F(Q)\) and define \(P \coloneq
\text{range}(\sigma)\). Since \(\Tf^{<n}(P)\) is finite for the finite qo
\(P\) (by \cref{lemma:tfn-finiteness}), the qo \(i_{\omega^n}^F(P)\) is also
finite (modulo equivalences) and trivially a wqo.

Since \(\text{factors}(\sigma) \subset i_{\omega^n}^F(P)\),
\(\text{factors}(\sigma)\) is a finite set. Let \(\text{factors}(\sigma)\) be the
set \(\Set{ \tau_0, \tau_1, \dots, \tau_{k-1} }\) indexed arbitrarily, for some \(k
< \omega\). Since \(g(\tau_i)\) is defined for each \(i < k\), \(g(\sigma)\) is also
defined and is a tree in \(\Tf^{=n}(P)\).

Let \(\tau \in i_{=\omega^n}^F(P)\) be the sequence \((\tau_0 + \tau_1 + \cdots +
\tau_{k-1})^\omega\). Like in the proof of \cref{prop:sequence-decomposition}, we
can decompose \(\sigma\) as a sequence \(\sigma_0 + \sigma_1 + \cdots\) of length
\(\omega\), where each \(\sigma_i \in i_{\omega^n}^F(P)\) for \(i < \omega\). Since
\(\sigma\) is an indecomposable sequence, \((\sigma_i)^\omega \preceq \sigma\) for
every \(i < \omega\), and thus \(\sigma_i \in \text{factors}(\sigma)\). Thus,
\(\sigma \preceq \tau\). It is straightforward to show that \(\tau \preceq \sigma\), and
thus \(\sigma \equiv \tau\). In other words, every sequence \(\sigma \in
i_{\omega^{n+1}}^F(Q)\) is uniquely identified by its finite set of cofinal
factors.

Now consider arbitrary sequences \(\sigma, \tau \in i_{\omega^{n+1}}^F(Q)\). We
need to show that \(\sigma \preceq \tau\) iff \(g(\sigma) \le_T g(\tau)\). If either
\(|\sigma| = 1\) or \(|\tau| = 1\), the equivalence follows from definition; so,
assume otherwise.

Let \(\sigma \equiv (\sigma_0 + \sigma_1 + \cdots + \sigma_{k-1})^\omega\) with
\(\sigma_i \in \text{factors}(\sigma)\) for \(i < k\); similarly, let \(\tau \equiv
(\tau_0 + \tau_1 + \cdots + \tau_{l-1})^\omega\) with \(\tau_j \in
\text{factors}(\tau)\) for \(j < l\). Pick any \(i < k\). If \(\sigma \preceq \tau\),
then \(\sigma_i \preceq (\tau_0 + \cdots + \tau_{l-1})^{m_i}\) for some \(m_i <
\omega\). By \cref{prop-atomic-indecomposable}, there is some \(j < l\) such that
\(\sigma_i \preceq \tau_j\). The converse clearly holds: if for every \(i < k\) there
is some \(j < l\) such that \(\sigma_i \preceq \tau_j\), then \(\sigma \preceq \tau\).

Since \(g \upharpoonright i_{\omega^n}^F(Q)\) is an order-embedding, \(\sigma_i
\preceq \tau_j\) iff \(g(\sigma_i) \le_T g(\tau_j)\) for any \(i,j\). Thus, by definition
of \(\le_T\), \(\sigma \preceq \tau\) iff \(g(\sigma) \le_T g(\tau)\). This completes the
proof.
\end{proof}

Since we have shown that both \(f\) and \(g\) are well-defined and order-embeddings
between \(i^F_{\omega^\omega}(Q)\) and \(\Tf(Q)\), we are finally ready to state the
main theorem of this section.

\begin{theorem}
The qos \((\Tf(Q), \le_T)\) of finite leaf-labelled trees over \(Q\) ordered by tree
homomorphisms, and \((i^F_{\omega^\omega}(Q), \preceq)\) of indecomposable
sequences over \(Q\) with finite range of length less than \(\omega^\omega\) ordered
by sequence embeddability are equivalent.
\label{thm:tree-sequence-correspondence}
\end{theorem}

\begin{proof}
By \cref{lemma:f-order-embedding} and \cref{prop:g-order-embedding}, the maps \(f
\colon \Tf(Q) \to i^F_{\omega^\omega}(Q)\) and \(g \colon i^F_{\omega^\omega}(Q)
\to \Tf(Q)\) are both well-defined order-embeddings. By
\cref{remark:qo-equivalent}, it suffices to show either of the two properties in
\cref{def:qo-equivalent}. Thus we show the following: for any \(\sigma \in
i^F_{\omega^\omega}(Q)\), \(f(g(\sigma)) \equiv \sigma\), by induction on the
length of \(\sigma\).

For \(\sigma \in i^F_{=1}(Q)\) the statement follows from definition. So, let
\(\sigma \in i^F_{=\omega^n}(Q)\) for \(n > 0\), with the statement assumed to hold
for any sequence in \(i^F_{\omega^n}(Q)\).

Let \(\text{factors}(\sigma) = \Set{ \sigma_0, \sigma_1, \dots, \sigma_{k-1} }\)
for some \(k < \omega\). By definition,
\begin{align*}
g(\sigma) &= \vertex{g(\sigma_0), g(\sigma_1), \dots, g(\sigma_{k-1})} \text{, and} \\
f(g(\sigma)) &= (f(g(\sigma_0)) + f(g(\sigma_1)) + \cdots + f(g(\sigma_{k-1})))^\omega.
\end{align*}

We know that \(\text{factors}(\sigma) \subset i^F_{\omega^n}\). By induction
hypothesis, for each \(i < k\), \(f(g(\sigma_i)) \equiv \sigma_i\). Thus,
\(f(g(\sigma)) \equiv (\sigma_0 + \sigma_1 + \cdots + \sigma_{k-1})^\omega\). In
the proof of \cref{prop:g-order-embedding} we showed that \(\sigma \equiv
(\sigma_0 + \sigma_1 + \cdots + \sigma_{k-1})^\omega\), and hence \(f(g(\sigma))
\equiv \sigma\) and we are done.
\end{proof}
\subsection{Maximal order types}
\label{subsec:maximal-order-types-sequences}
With the help of the correspondence between leaf-labelled trees and
indecomposable sequences, we can simply reuse the maximal order types which were
proved in \cref{thm:Tf-order-type}.

\begin{theorem}
Given a non-empty wqo \((Q,\le)\), \(o(i^F_{\omega^\omega}(Q)) = o(\Tf(Q)) =
g(o(Q))\), where \(g(1) = \omega\) and \(g(\alpha) = \bar\varepsilon(-2 + \alpha)\)
for \(\alpha \geq 2\).
\label{thm:iww-order-type}
\end{theorem}

The maximal order type for \emph{arbitrary} transfinite sequences is obviously bounded
below by the maximal order type of indecomposable sequences. From
\cref{prop:sequence-decomposition}, we can derive a loose upper bound
\(o(s^F_{\omega^n}(Q)) \leq o(i^F_{\omega^n}(Q)^*)\) for wqo Q and finite \(n <
\omega\). However, we can do better if we collectively consider all sequences of
length less than \(\omega^\omega\).

\begin{proposition}
For every non-empty wqo \((Q,\le)\), \(o(s^F_{\omega^\omega}(Q)) =
o(i^F_{\omega^\omega}(Q))\).
\label{prop:sww-iww-same-type}
\end{proposition}

\begin{proof}
First, a trivial case: when \(|Q| = o(Q) = 1\), clearly \(s^F_{\omega^\omega}(Q) =
i^F_{\omega^\omega}(Q) = \omega\). The rest of the proof assumes that \(o(Q) > 1\).

Every sequence \(\sigma \in s^F_{\omega^\omega}(Q)\) is embeddable in some
indecomposable sequence \(\sigma_{\mathit{in}} \in i^F_{\omega^\omega}(Q)\): in
particular, it holds for \(\sigma_{\mathit{in}} = \sigma^\omega\). We say that
\(i^F_{\omega^\omega}(Q)\) is a \emph{majorizing} subset of \(s^F_{\omega^\omega}(Q)\).

To simplify notation, for any \(\sigma \in s^F_{\omega^\omega}(Q)\), let
\(L(\sigma) \coloneq L_{s^F_{\omega^\omega}(Q)}(\sigma)\); and when \(\sigma \in
i^F_{\omega^\omega}(Q)\), let \(L_{\mathit{in}}(\sigma) \coloneq
L_{i^F_{\omega^\omega}(Q)}(\sigma)\).

For any \(\sigma_0, \sigma_1 \in s^F_{\omega^\omega}(Q)\), if \(\sigma_0 \preceq
\sigma_1\) then \(L(\sigma_0) \subset L(\sigma_1)\) and \(o(L(\sigma_0)) \leq
o(L(\sigma_1))\). Thus, we get:
\begin{align*}
\label{eq:1}
o(s^F_{\omega^\omega}(Q))
&= \sup_{\sigma \in s^F_{\omega^\omega}(Q)} o(L(\sigma)) \\
&= \sup_{\sigma_{\mathit{in}} \in i^F_{\omega^\omega}(Q)} o(L(\sigma_{\mathit{in}})). \tag{1}
\end{align*}

We are not done yet, since \(L(\sigma_{\mathit{in}})\) in the equation is with
respect to the qo \(s^F_{\omega^\omega}(Q)\) and not
\(i^F_{\omega^\omega}(Q)\).

Pick an arbitrary \(\sigma_{\mathit{in}} \in i^F_{\omega^\omega}(Q)\) and consider
any \(\tau \in s^F_{\omega^\omega}(Q)\). By \cref{prop:sequence-decomposition}, we
know that we can decompose \(\tau\) as \(\tau = \tau_0 + \tau_1 + \cdots +
\tau_{k-1}\) where every \(\tau_i \in i^F_{\omega^\omega}(Q)\) for \(i < k <
\omega\). If \(\sigma_{\mathit{in}} \preceq \tau\) then from
\cref{prop-atomic-indecomposable} there is some \(i < k\) such that
\(\sigma_{\mathit{in}} \preceq \tau_i\). Thus, if \(\sigma_{\mathit{in}}
\not\preceq \tau\) then \(\sigma_{\mathit{in}} \not\preceq \tau_i\) for all \(i <
k\). Thus: $$ L(\sigma_{\mathit{in}}) \quasiembedding
(L_\mathit{in}(\sigma_{\mathit{in}}))^*. $$

By definition, \(o(L_\mathit{in}(\sigma_{\mathit{in}})) <
o(i^F_{\omega^\omega}(Q))\). Since we assumed that \(o(Q) > 1\),
\(o(i^F_{\omega^\omega}(Q))\) is an epsilon number and hence closed under
\(\omega\text{-exponentiation}\). Thus,
\begin{align*}
o(L(\sigma_{\mathit{in}}))
&\leq o(L_\mathit{in}(\sigma_{\mathit{in}}))^* \\
&\leq \omega^{\omega^{o(L_\mathit{in}(\sigma_{\mathit{in}})) + 1}} \\
&< o(i^F_{\omega^\omega}(Q))
\end{align*}

The bounds for \(o(Q^*)\) (i.e., the Higman order over \(Q\)) can be found in
\cite[Theorem 2.9]{DianaSchmidt2020}. This shows that $$
\sup_{\sigma_{\mathit{in}} \in i^F_{\omega^\omega}(Q)}
o(L(\sigma_{\mathit{in}})) \leq o(i^F_{\omega^\omega}(Q)) $$ and, by \cref{eq:1},
we are done.
\end{proof}

Let us restate the results for the maximal order type of
\(s^F_{\omega^\omega}(Q)\) in isolation.

\begin{theorem}
Given a non-empty wqo \((Q,\le)\), \(o(s^F_{\omega^\omega}(Q)) = g(o(Q))\), where
\(g(1) = \omega\) and \(g(\alpha) = \bar\varepsilon(-2 + \alpha)\) for \(\alpha \geq
2\) (where \(\bar\varepsilon\) is the fixed-point-free epsilon function defined in
\cref{def:fixed-point-free-epsilon-function}).
\label{thm:sww-order-type}
\end{theorem}
\section{Related Work and Further Questions}
\label{sec:further-questions}
This work coincides with a spurt of interest in the maximal order types of wqos,
particularly to resolve questions about qos on transfinite sequences, and more
recently defined qos on trees. We mention two very closely related works. Altman
\cite{Altman2024} considers the proof by Erdős and Rado for transfinite
sequences (with finite range) of length strictly less than \(\omega^\omega\), and
derives bounds for the maximal order types of \(s^F_{\alpha}(Q)\) for any \(\alpha
< \omega^\omega\) and wqo \(Q\). Friedman and Weiermann \cite{Friedman2023}
consider finite trees with labels on \emph{both} internal and leaf vertices, also
ordered by tree homomorphisms, and prove bounds for their maximal order types.
The bounds are slightly loose, but suffice for their focus on independence
results. As mentioned before, we sharpen their results when restricted to
leaf-labelled trees with very similar -- but more careful -- proofs. The
extension to trees with internal labels will appear in an upcoming work by the
first author.

There have also been recent advancements on questions at the intersection of
reverse mathematics and the theory of so-called \emph{better-quasi-orders}. This work
is no exemption. For instance, Freund \cite{Freund2023} uses lower bounds on
the maximal order type of \(H_f(Q)\) (for wqo \(Q\)) to show an implication from a
certain statement on better-quasi-orders to \(\mathsf{ACA^+_0}\). The qo \(\Tf(Q)\)
is essentially a ``height-preserving'' strengthening of \(H_f(Q)\) (and trivially
\(H_f(Q) \quasiembedding \Tf(Q)\)), and can be used to rephrase a simpler version
of \cite[Theorem 3.8 and Corollary 3.9]{Freund2023}. The second author, along
with Soldà \cite{Pakhomov2024}, has also worked on the reverse mathematical
strength of Nash-Williams' theorem, closing a major open problem in the field.
The correspondence between the qos \(\Tf(Q)\) and \(i^F_{\omega^\omega}(Q)\) is
closely related to the connection between -- in the terminology of
\cite{Pakhomov2024} -- \(\dot{V}_{\omega}(Q)\) and indecomposable sequences.

There are a few extensions to the qo \(\Tf(Q)\) that are of interest. The simplest
one is with internal labels, and as mentioned before, was already considered in
\cite{Friedman2023} and due to appear in future work by the first author. The
extension to arbitrary well-founded trees (i.e., not necessarily of finite
height) is of significant interest, for instance in the context of
\cite{Montalbn2006}, and has proved quite difficult to tackle. Yet another
promising direction is to introduce the so-called \emph{gap-condition} to the
definition of tree homomorphisms, analogous to Friedman's extension to Kruskal's
tree ordering, and to extract independence results and ordinal notation systems
for arithmetical theories stronger than \(\mathsf{ATR_0}\).
\section*{Acknowledgements}
\label{sec:acknowledgements}
The first author would like to thank Andreas Weiermann for the introduction to
-- and subsequent encyclopedic knowledge on -- the theory of well-quasi-orders
and maximal order types; and an anonymous referee for their various helpful
comments.
\subsection*{Funding}
\label{sec:funding}
The first author was funded by the \emph{Bijzonder Onderzoeksfonds} (BOF)
grant BOF.DOC.2021.0077.02. The work of the second author is supported by the
FWO-Odysseus project G0F8421N.
\bibliographystyle{amsplain}
\bibliography{main}
\end{document}